\pdfoutput=1
\documentclass[12pt]{paper}

\usepackage{color}
\usepackage{mathptmx}       % selects Times Roman as basic font
\usepackage{helvet}         % selects\textbf{\textbf{â¢}} Helvetica as sans-serif font
\usepackage{courier}        % selects Courier as typewriter font
\usepackage{type1cm}        % activate if the above 3 fonts are
\usepackage{makeidx}         % allows index generation
\usepackage{graphicx}        % standard LaTeX graphics tool
\usepackage{multicol}        % used for the two-column index
\usepackage[bottom]{footmisc}% places footnotes at page bottom
\usepackage{bbm}
\usepackage{amsmath}
\usepackage{amsfonts}
\usepackage{amssymb, eucal}
\usepackage{amsthm}
\usepackage{comment}

\usepackage[colorlinks,citecolor = red, linkcolor=blue,hyperindex]{hyperref}

\usepackage{times,upref, mathrsfs, xcolor, dsfont, extarrows}
\usepackage{amsfonts,amsmath,amstext,amsbsy, amsopn}

\usepackage{amscd,pifont}
\usepackage[curve]{xypic}
\usepackage{calligra}
\usepackage{float}
 \usepackage{color}

\newtheorem{theorem}{Theorem}
\newtheorem{corollary}[theorem]{Corollary}

\newtheorem{lemma}[theorem]{Lemma}
\newtheorem*{conjecture}{The Lyons-Peres Conjecture}
\newtheorem*{conj}{Conjecture}
\newtheorem*{conj-GAF}{Lyons-Peres conjecture for GAF zeros}
\newtheorem{proposition}[theorem]{Proposition}

\newtheorem{remark}[theorem]{Remark}

\newtheorem{definition}[theorem]{Definition}

\newtheorem*{definition*}{Definition}
\newtheorem*{convention*}{Convention}

\newtheorem*{remark*}{Remark}
\newtheorem*{theorem*}{Theorem}

\newtheorem*{claim*}{Claim}
\newtheorem*{notation*}{Notation}
\usepackage[top=2cm, bottom=2cm, left=2cm, right=2cm]{geometry}
\numberwithin{theorem}{section}
\numberwithin{figure}{section}
\numberwithin{equation}{section}

\begin{document}

\title{Kernels of  conditional determinantal measures and the proof of the Lyons-Peres Conjecture}

\author{Alexander I. Bufetov, Yanqi Qiu and Alexander Shamov}
%\thanks{ }

 %\authorrunning{A. I. Bufetov, Y. Qiu and A. Shamov}
% Use for an abbreviated version of
% your contribution title if the original one is too long%

\maketitle

\abstract{The main result  of this paper, Theorem \ref{mainthm-DT}, establishes a conjecture of Lyons and Peres: for a determinantal point process governed by a reproducing kernel, the system of kernels sampled at the particles of a random configuration is complete in the range of the  kernel.  A key step in the proof, 
Lemma \ref{key-lem}, states  that 
conditioning on  the configuration in a subset preserves the determinantal property, and the main Lemma \ref{main-local} is a new local property for  kernels of  conditional  point processes.
 In Theorem \ref{main-thm3} we prove the triviality of  the tail sigma-algebra  for determinantal  point processes governed by self-adjoint kernels. 
 }

\newcommand{\R}{\mathbb{R}}
\newcommand{\N}{\mathbb{N}}
\newcommand{\Z}{\mathbb{Z}}
\newcommand{\Q}{\mathbb{Q}}
\newcommand{\D}{\mathbb{D}}
\newcommand{\C}{\mathbb{C}}
\newcommand{\T}{\mathbb{T}}
\newcommand{\E}{\mathbb{E}}
\newcommand{\PP}{\mathbb{P}}
\newcommand{\Prob}{\mathbb{P}}

\newcommand{\F}{\mathbb{F}}

\newcommand{\X}{\mathscr{X}}
\newcommand{\M}{\mathcal{M}}
\newcommand{\ZZ}{\mathcal{Z}}
\newcommand{\K}{\mathcal{K}}
\newcommand{\Leb}{\mathrm{Leb}}
\newcommand{\Conf}{\mathrm{Conf}}
\newcommand{\Tr}{\mathrm{Tr}}
\newcommand{\tr}{\mathrm{tr}}
\newcommand{\rank}{\mathrm{rank}}
\newcommand{\spec}{\mathrm{spec}}
\newcommand{\Var}{\mathrm{Var}}
\newcommand{\supp}{\mathrm{supp}}
\newcommand{\spann}{\mathrm{span}}
\newcommand{\esssup}{\mathrm{esssup}}
\newcommand{\sgn}{\mathrm{sgn}}

\newcommand{\Det}{\mathrm{det}}
\newcommand{\Ran}{\mathrm{Ran}}

\newcommand{\an}{\text{\, and \,}}
\newcommand{\as}{\text{\, as \,}}

\newcommand{\ch}{\mathbbm{1}}

\newcommand{\dd}{\mathrm{d}}
\newcommand{\Aut}{\mathrm{Aut}}

\renewcommand{\restriction}{\mathord{\upharpoonright}}

\section{Introduction}
\subsection{The zero set of a Gaussian Analytic Function on the disc is a uniqueness set for the Bergman space}
Consider the random series 
\begin{equation}\label{gaus-series}
\frak{f}_\D(z) = \sum_{n=0}^\infty f_n z^n,
\end{equation}
where the coefficients $f_n$ are independent identically distributed complex Gaussian random variables with expectation $0$ and variance $1$. 
The series (\ref{gaus-series}) has  radius of convergence $1$  almost surely and defines a holomorphic function on the open unit disc $\D$. 
Let  $Z(\frak{f}_\D)$ be the zero set of (\ref{gaus-series}):
\begin{align}\label{def-GAF}
Z(\frak{f}_\D) = \{z \in \D: \frak{f}_\D(z) =0\}.
\end{align}

Denote $A^2(\D)$ the Bergman space of  holomorphic functions on $\D$ square-integrable with respect to the Lebesgue measure $\Leb$.
A subset $X\subset \D$ is called a {\it uniqueness set} for $A^2(\D)$ if a function $h \in A^2(\D)$ satisfying $h\restriction_X = 0$ must be the zero function.  In this particular case, our main result is
\begin{theorem}\label{thm-Zf}
Almost surely, $Z(\frak{f}_\D)$ is a  uniqueness set for $A^2(\D)$. 
\end{theorem}
In other words, almost surely, $Z(\frak{f}_\D)$ can not be  a zero set of a function in $A^2(\D)$. Theorem \ref{thm-Zf} is a direct corollary of our main result, Theorem \ref{mainthm-DT}, formulated  below, 
 since,  by the  Peres-Vir{\'a}g Theorem \cite{PV-acta}, the random subset $Z(\frak{f}_\D) \subset \D$ is a realization of the determinantal point process on $\D$ governed by the reproducing kernel 
\[
K_\D (z, w) = \frac{1}{\pi (1 - z \bar{w})^2}
\]
 of the subspace $A^2(\D)\subset L^2(\D, \Leb)$.

\begin{remark}
After the work on this paper was finished, we became aware of the result of Lyons and Zhai \cite{Lyons-Zhai}, where they prove in particular in a different way that $Z(\mathfrak{f}_\D)$ is almost surely a uniqueness set for $A^2(\D)$. 
\end{remark}

For brevity, the set of zeros of a nonzero function in $A^2(\D)$ will be called an $A^2(\D)$-zero set. Various necessary and sufficient conditions for a subset of the disc
to be an $A^2(\D)$-zero set exist in the literature. For example, by
Theorem 4.7 in  Hedenmalm-Korenblum-Zhu \cite{HKZ-Bergman}, for any  $A^2(\D)$-zero set  $Z$  
 the Blaschke-type condition
\begin{align}\label{B-cond}
\sum\limits_{z \in Z, z\neq 0}  \frac{1 - | z|}{\Big[\log \frac{1}{1 - |z|}\Big]^{1+\varepsilon}}<\infty
\end{align}
holds for any $\varepsilon> 0$. Note, however, that for our random set $Z(\frak{f}_\D)$, using the determinantal structure of $Z(\frak{f}_\D)$, we have
\[
\E \Big(\sum_{z \in Z(\frak{f}_\D)} \frac{1 -|z|}{\big[  \log \frac{1}{1 - |z|}\big]^{1 + \varepsilon}} \Big)=  \int_\D  \frac{1 -|z|}{\big[  \log \frac{1}{1 - |z|}\big]^{1 + \varepsilon}} \frac{1}{\pi (1 - |z|^2)^2} dA(z) < \infty, 
\]
 hence $Z(\frak{f}_\D)$ satisfies  \eqref{B-cond} almost surely. Indeed,  $A^2(\D)$-zero sets cannot be described in terms of conditions involving radii alone, the angular distribution of the zeros plays a role.
 
We next observe that $Z(\frak{f}_\D)$ is neither a sampling nor an interpolating set for $A^2(\D)$. 
 Recall that a  discrete subset $Z\subset \D$ is called an $A^2(\D)$-sampling set if there exist positive constants $C_1, C_2>0$ such that for any $g \in A^2(\D)$
we have 
\[
C_1 \|g \|_{A^2(\D)}^2 \le \sum\limits_{z\in Z} | g (z)|^2 (1 -|z|^2)^2 \le C_2 \| g \|_{A^2(\D)}^2.
\]
By definition, an $A^2(\D)$-sampling set is also an $A^2(\D)$-uniqueness set.
A discrete subset  $Z\subset \D$, $Z=\{z_1, \dots, z_j, \dots\}$  is called an $A^2(\D)$-interpolating set if for any sequence $\{a_j\}$ in $\C$ such that $\{a_j (1 - |z_j|^2)\} \in \ell^2$, there exists $g\in A^2(\D)$ such that $g(z_j) = a_j$ for all $j$. Deleting a finite number of points from any uniqueness set for $A^2(\D)$ does not change the uniqueness property of the set, consequently a function in $A^2(\D)$ can not vanish at all points of a uniqueness set except at a finite subset, which means that  a uniqueness set for $A^2(\D)$ is never $A^2(\D)$-interpolating. 

\begin{proposition}\label{sample-inter}
The subset $Z(\frak{f}_\D)$ is almost surely neither $A^2(\D)$-sampling nor $A^2(\D)$-interpolating.
\end{proposition}
To see that the set $Z(\frak{f}_\D)$ is not sampling, recall Seip's theorem, Theorem 7.1 in \cite{Seip-d-zero},  stating that any $A^2(\D)$-sampling set can be expressed as a finite union of sets uniformly separated with respect to the Lobachevskian distance. 
\begin{lemma}\label{non-unif-sep}
Almost surely, $Z(\frak{f}_\D)$ can not be expressed as a finite union of uniformly separated sets. 
\end{lemma}

Lemma \ref{non-unif-sep}, proved in Section \ref{sample} below with the use of mixing, under the measure $\Prob_{K_{\D}}$, of the action of one-parametre groups of isometries of the Lobachevsky plane, implies Proposition \ref{sample-inter}.

\subsection{An outline  of the main results}
\subsubsection{The Lyons-Peres Conjecture}
 Let $E$ be a locally compact  $\sigma$-compact Polish space, let $\Conf(E)$ be the space of  locally finite configurations on $E$.  Let  $\mu$ be a sigma-finite Radon measure on $E$,  let $K$ be the kernel of a locally trace class positive contraction acting in the complex Hilbert space $L^2(E, \mu)$, and let $\Prob_K$ be the corresponding determinantal measure on $\Conf(E)$ (the precise definitions are recalled in the next subsection).

Let $K$ be a locally trace class orthogonal projection onto a closed subspace  $H$ of $L^2(E,\mu)$; in other words,  let $H\subset L^2(E, \mu)$ be a  reproducing kernel Hilbert space, and let  $K$ be the reproducing kernel for $H$.
For $x\in E$, introduce a function $K_x\in L^2(E, \mu)$ by the formula
\begin{equation}\label{kxt}
K_x(t) := K(t,x), \ t\in E.
\end{equation}
Our main result,  Theorem \ref{mainthm-DT}, establishes 
\begin{conjecture}
For  $\PP_K$-almost every $X \in \Conf(E)$, we have
\begin{align}\label{lyons-equality}
\overline{\spann}^{L^2(E, \mu)}  \{K_x; x\in X \}=H.
\end{align}
\end{conjecture}
  Lyons \cite[Theorem 7.11]{DPP-L}  proved that the completeness of reproducing kernels holds when $E$ is  countable and formulated the general statement as 
  Conjecture 4.6 in \cite {Lyons-ICM}. Ghosh \cite{Ghosh-sine}  established the conjecture  under the important additional assumption that the  determinantal point process $\PP_K$ is number rigid in the sense of Ghosh and Peres.   
  While many determinantal point processes are indeed number rigid  (cf.  Ghosh \cite{Ghosh-sine} for the sine-process, 
  Ghosh and Peres \cite{Ghosh-rigid} for the Ginibre ensemble,  \cite{Buf-rigid} for processes governed by the Airy, the Bessel and more general integrable kernels, \cite{BDQ-ETDS} for stationary processes, \cite{BQ-CMP} for generalized Ginibre ensembles),
 our zero set $Z(\frak{f}_\D)$ is not: indeed, Holroyd and Soo \cite{Soo} showed that the point process $Z(\frak{f}_\D)$ is insertion and deletion tolerant,  the opposite  of being number rigid. For determinantal point processes associated with generalized Bergman spaces on $\D$,  insertion and deletion tolerance is established in \cite{BQ-CMP} and the Radon-Nikodym derivative 
of the Palm measure with respect to the initial measure is given explicitly as a generalized multiplicative functional.

\subsubsection{Kernels of conditional  determinantal point processes}

The key steps in our proof of the Lyons-Peres Conjecture are the preservation of the determinantal structure under conditioning with respect to the configuration in a subset  and  a new {\it local property}
for conditional kernels of determinantal point processes. 

Given  a Borel probability measure $\PP$ on $\Conf(E)$ and a Borel subset $C\subset E$, the measure $\mathbb{P}(\cdot | X; C)$ on the space $\Conf(E\setminus C)$ is defined as the conditional measure of $\mathbb{P}$ with respect to the condition that the restriction of our random configuration onto $C$ coincides with $X\cap C$ (see \S 2 below for the detailed definition).

Lemma \ref{key-lem} establishes  that, for any determinantal point process $\PP_K$ induced by a self-adjoint locally trace class kernel $K$,  the conditional measures $\Prob_K(\cdot | X; C)$ are themselves determinantal and governed by  explicitly given  self-adjoint kernels.
 For {\it precompact} subset $B\subset E$, the determinantal property for $\Prob_K(\cdot | X; B)$ follows from the characterization of Palm measures for determinantal processes due to Shirai-Takahashi \cite{ST-palm} and the characterization of induced determinantal processes \cite{Buf-umn}, \cite{Buf-inf}.
For $X\in \Conf(E)$,
in Definition \ref{def-canonical-kernel} below we introduce  a specific
self-adjoint  kernel $K^{[X, \, B]}$ governing the measure $\mathbb{P}_K(\cdot |X; B)$.

In order to prove that conditioning preserves the determinantal property, we shall show that, along an increasing or a decreasing sequence of precompact subsets $B_n$,  our specifically chosen kernels $K^{[X, \, B_n]}$ form a martingale after a suitable compression.
The one-step martingale property (corresponding to the case of two precompact subsets $B_0 = \emptyset$ and $B_1 = B$) for spanning trees is due to Benjamini, Lyons, Peres and Schramm
 \cite{BLPS} and  for processes on general discrete phase spaces to Lyons \cite{DPP-L}. It seems to be essential for the argument of Benjamini, Lyons, Peres and Schramm  \cite{BLPS}, Lyons \cite{DPP-L} that the phase space be discrete;
we do not see how to extend their argument to continuous phase spaces. Moreover, it requires some effort to prove fully the martingale property from the one-step martingale property.

Our proof of the martingale property of the conditional kernels relies on a new local property for the kernels $K^{[X, \, B]}$ which we now informally explain.
 If $B\subset C\subset E$, then conditioning
on the restriction of the configuration onto $B$ commutes with the natural projection map $X\to X\cap C$ from $\Conf(E)$ to $\Conf(C)$. This commutativity manifests itself on the level of the  kernels chosen in Definition \ref{def-canonical-kernel} below:
we have
$
\chi_CK^{[X\cap B,\, B]}\chi_C=(\chi_CK\chi_C)^{[X\cap C, \, B]}.
$
Our local property claims that instead of $\chi_C$
one can take a much more general projection $Q$, and the relation still holds.
More precisely, let  $Q: L^2(E, \mu) \rightarrow L^2(E,\mu)$ be an orthogonal projection such that  $\Ran(Q) \subset L^2(E\setminus B, \mu)$ and
that $QKQ$  is locally trace-class. In Lemma \ref{main-local} below we shall see that
\begin{equation}\label{local-kernel}
\Big((Q + \chi_B)K (Q+\chi_B)\Big)^{[X, \, B]} = (Q  + \chi_B )K^{[X, \, B]} (Q  + \chi_B) =  Q K^{[X, \, B]} Q.
\end{equation}

 Applying  (\ref{local-kernel}) to a  one-dimensional projection operator $Q$, we obtain that, for an arbitrary $\varphi\in L_2(E\setminus B, \mu)$,  the quantity $\langle K^{[X,B]}\varphi, \varphi\rangle$ is a martingale indexed by $B$ (with respect to the partial order of inclusions of sets), cf. (\ref{exp-kxb})  below.
Using the {\it Radon-Nikodym property} for the space of trace-class operators, we obtain an operator-valued martingale that converges, along an increasing sequence of precompact subsets of $E$, almost surely in the space of locally trace-class operators. As an immediate consequence, we prove that for determinantal point processes governed by self-adjoint kernels, conditioning on the configuration in any Borel subset, preserves the determinantal property, see Lemma  \ref{key-lem}.

\subsubsection{Triviality of the  tail $\sigma$-algebra}
As an application of the local property for the conditional kernels, in Theorem \ref{main-thm3}, we  establish the triviality of the tail sigma-algebra for determinantal point processes governed by self-adjoint kernels. Lyons proved   tail-triviality in the discrete setting in \cite{DPP-L},
extending the argument of Benjamini-Lyons-Peres-Schramm \cite{BLPS} for spanning trees  and conjectured that tail triviality holds in full generality \cite[Conjecture 3.2]{Lyons-ICM}.
The argument of Benjamini-Lyons-Peres-Schramm \cite{BLPS} and of Lyons \cite{DPP-L} relies on an estimate for the decay of the variance of the conditional kernel;
using the local property of Lemma \ref{main-local}, we establish a similar  variance estimate in full generality, see Lemma  \ref{lem-var}, and obtain the desired triviality of the tail sigma-algebra. The local property of conditional kernels thus allows us  to carry out the proof  of tail triviality in a unified way for both the continuous and the discrete   setting.

 The triviality of the tail sigma-algebra for general determinantal point processes with self-adjoint kernels is  the main result of the independent  work by Osada and Osada \cite{Osada-Osada}. The argument of Osada-Osada \cite{Osada-Osada} is completely different from ours: Osada-Osada \cite{Osada-Osada}
construct a special family of discrete approximations of  continuous determinantal point processes  and derive the triviality of the tail sigma-algebra in the continuous setting from the theorem of Lyons by approximation.   Another approach, due to Lyons [private  communication], for establishing the triviality of the tail sigma-algebra in the
continuous setting, also proves it from  the  discrete result using the Goldman's transference principle, cf.  Goldman \cite[Proposition 12]{Goldman} and Lyons \cite[Section 3.6]{Lyons-ICM}.

 \subsection{Formulation of the main results}
Let $E$ be a locally compact  $\sigma$-compact Polish space, equipped with a metric such that any bounded set is relatively compact, and  endowed with a positive $\sigma$-finite Radon measure $\mu$. Let  $\Conf(E)$ be the space of  locally finite configurations on $E$. A point process on $E$ is by definition a Borel probability measure on $\Conf(E)$. Let  $K$ be   a bounded self-adjoint {\it locally trace class} operator $K: L^2(E, \mu)\rightarrow L^2(E, \mu)$ such that the spectrum $\spec(K) \subset [0, 1]$.  A theorem obtained by Macchi \cite{DPP-M} and  Soshnikov \cite{DPP-S}, as well as by Shirai and Takahashi  \cite{ShirTaka0}, gives a unique point process on $E$, denoted by  $\PP_{K}$, such that  for any compactly supported bounded measurable function $g: E \rightarrow \C$, we have
\begin{align*}
\E_{\PP_K} \Big[\prod_{x\in X} (1 + g(x)) \Big]   = \det  \Big( 1 +   \sgn(g)   |g |^{1/2}  \cdot K \cdot |g|^{1/2}\Big)_{L^2(\mu)}, \quad \sgn(g) =\frac{g}{|g|}.
\end{align*}
Here $\det (1 + S)$ denotes the Fredholm determinant of the operator $1+S$, see, e.g., Simon \cite{Simon-det}.

The locally trace class  self-adjoint operator  $K$ is an integral operator. Following Soshnikov \cite{DPP-S}, we  fix  a Borel subset  $E_0\subset E$ with $\mu(E \setminus E_0) =0$  and {\it fix} a Borel function $K: E_0 \times E_0 \rightarrow \C$, our kernel, in such a way that for any $k\in\N$ and any bounded Borel subset $B\subset E$, we have
\begin{align}\label{kernel-ass}
\tr ((\chi_B K \chi_B)^k) = \int_{B^k} K(x_1, x_2) K(x_2, x_3)\cdots K(x_k, x_1) \dd \mu(x_1) \cdots \dd \mu(x_k).
\end{align}

\begin{theorem}\label{mainthm-DT}
Let $K$ be a locally trace-class orthogonal projection onto a subspace  $H$ of $L^2(E,\mu)$.
Then for $\PP_K$-almost every $X \in \Conf(E)$, the functions $K_x$ defined by (\ref{kxt}) satisfy
\[
\overline{\spann}^{L^2(E, \mu)}  \{K_x; x\in X \}=H.
\]

\end{theorem}

If we fix a realization for each $h\in H$ in such a way that  the equation
$h(x) = \langle h, K_x\rangle$ holds  for every $x\in E_0$ and every $h\in H$, then Theorem \ref{mainthm-DT} can equivalently be reformulated as follows:
\begin{corollary}\label{zeroset}
For $\PP_K$-almost every $X \in \Conf(E)$,  if $h\in H$ satisfies $h\restriction_{X} =0$, then $h =0$.
\end{corollary}

\begin{theorem}\label{main-thm3}
Let $B_1 \subset \cdots \subset B_n \subset \cdots \subset E$ be an increasing exhausting sequence of bounded Borel subsets of $E$. The $\sigma$-algebra
$
\bigcap_{n\in\N}\mathcal{F}(E\setminus B_n)
$
is trivial with respect to $\PP_K$.
\end{theorem}

\begin{corollary}\label{cor-Lyons-conj}
The point process $\PP_K$  has trivial tail $\sigma$-algebra.
\end{corollary}

\begin{remark}
Our assumption on $\sigma$-compactness of $E$ is not essential: in the argument below, one could  everywhere  replace ``relatively compact'' (here equivalent to ``bounded") by ``having finite weight with respect to the measure  $K(x, x)d\mu(x)$''. On the other hand, the assumption of {\it self-adjointness} is used throughout. It would be interesting to 
obtain similar results for more general determinantal kernels.
\end{remark}
\subsubsection{The key lemma}

\begin{definition}\label{def-canonical-kernel}
For any bounded Borel subset $B\subset E$, we define canonical conditional kernels $K^{[X, \, B]}$ with respect to the conditioning on the configuration in $B$ as follows:
\begin{itemize}
\item For   $p\in E_0$, define a kernel $K^p$, for  $(x, y)\in E_0 \times E_0$, by the formula
\begin{align*}
K^{p}(x, y): = \left\{  \begin{array}{cc}  K(x, y) - \displaystyle \frac{K(x, p) K(p, y)}{K(p, p)} &  \text{if $K(p, p)> 0$}
\\
 0 &  \text{if $K(p, p)=0$}
 \end{array}.
 \right.
\end{align*}
\item For an $n$-tuple $(p_1, \cdots, p_n) \in E_0^n$,  define a kernel  $K^{p_1, \cdots, p_n} = (\cdots (K^{p_1})^{p_2} \cdots)^{p_n}$ as follows (cf. Shirai-Takahashi \cite[Corollary 6.6]{ST-palm}).
Given $x, y\in E_0$, write $p_0= x, q_0 = y$, $q_i = p_i$ for $1\le i \le n$, and set
 \begin{align}\label{def-Palm-kernel}
 K^{p_1, \cdots, p_n}(x, y) : =
 \left\{  \begin{array}{cl}
 \displaystyle \frac{\Det[K(p_i, q_j)]_{0 \le i, j \le n}
}{
 \Det[K(p_i, p_j)]_{1\le i, j \le n}
} &  \text{if $\Det[K(p_i, p_j)]_{1\le i, j \le n}>0$}
\\
 0 &   \text{if $\Det[K(p_i, p_j)]_{1\le i, j \le n} = 0$}
 \end{array}
 \right.
\end{align}
\item For a bounded Borel subset $B\subset E$ and $X \in \Conf(E)$  such that $X\cap B=\{p_1, \dots, p_l\} \subset E_0$,  define
\begin{align}\label{def-cond-kernel}
K^{[X,\, B]}=\left\{
 \begin{array}{cl}
\chi_{E\setminus B} K^{p_1, \dots, p_l}(1-\chi_B K^{p_1, \dots, p_l})^{-1}\chi_{E\setminus B} &  \text{if $1-\chi_B K^{p_1, \dots, p_l}$ is invertible}
\\
0 & \text{if $1-\chi_B K^{p_1, \dots, p_l}$ is not invertible}
 \end{array}.
 \right.
\end{align}
\end{itemize}
\end{definition}

We will see later, from the inequalities \eqref{both-side-c} and \eqref{one-side-c},  that if $1-\chi_B K^{p_1, \dots, p_l}$ is invertible, then the operator $\chi_B K^{p_1, \dots, p_l}$ is strictly contractive. Therefore, the series
\[
K^{[X,\, B]} = \chi_{E\setminus B}  \sum_{n=0}^\infty K^{p_1, \dots, p_l}(\chi_B K^{p_1, \dots, p_l})^{n}\chi_{E\setminus B}
\]
converges in the operator norm topology. In particular, for $(x, y) \in E_0 \times E_0$, we will use the formula
\begin{align}\label{canonical-KXB}
\begin{split}
K^{[X,\, B]}(x, y) =& \chi_{E\setminus B} (x) \chi_{E\setminus B} (y)  K^{p_1, \dots, p_l}(x, y)
\\
& + \chi_{E\setminus B} (x) \chi_{E\setminus B} (y)  \Big \langle     \Big( \sum_{n=1}^\infty  (\chi_B  K^{p_1, \dots, p_l})^{n-1}  \Big)  (\chi_B(\cdot) K^{p_1, \dots, p_l}(\cdot, y)),\, K^{p_1, \dots, p_l}(\cdot, x)  \Big \rangle_{L^2(E, \mu)}
\end{split}
\end{align}
as our specific Borel realization of the kernel  for the operator $K^{[X, \, B]}$.

\begin{remark*}
We will see in Proposition \ref{prop-DPP} below that $K^{[X, \, B]}$ is the correlation kernel for the conditional measure of $\PP_K$, the condition being that the configuration  on $B$ coincides with $X\cap B$. In particular, for $\PP_K$-almost every $X$, we have $X \cap B = \{p_1, \cdots, p_l\} \subset E_0$ and $1-\chi_B K^{p_1, \dots, p_l}$ is invertible. The second case $K^{[X, B]} =0$ has probability zero.  Note that the range of $K^{[X, \, B]}$ is contained in $L^2(E \setminus B, \mu)$ and we have
\[
K^{[X, \, B]} = \chi_{E\setminus B} K^{[X, \, B]} \chi_{E\setminus B}.
\]
\end{remark*}

 For any Borel subset $W\subset E$, {\it not necessarily bounded},  consider the Borel surjection $\pi_W: \Conf(E) \rightarrow \Conf(W)$ given by $X \mapsto X \cap W$. Fibres of this mapping can be identified with $\Conf(E\setminus W)$.  For  a  Borel probability measure $\PP$ on $\Conf(E)$,  the measure $\mathbb{P}(\cdot | X; W)$ on $\Conf(E\setminus W)$ is defined as the conditional measure of $\mathbb{P}$ with respect to the condition that the restriction of our random configuration onto $W$ coincides with $\pi_W(X)$. More formally,  the measures $\mathbb{P}(\cdot | X; W)$ are conditional measures, in the sense of Rohlin \cite{Rohlin-meas-eng},  of our initial measure $\PP$ on fibres of the measurable partition induced by the surjection $\pi_W$.

Denote by $\mathscr{L}_1(L^2(E, \mu))$ the space of trace class operators on $L^2(E, \mu)$ and by  $\mathscr{L}_{1, loc}(L^2(E, \mu))$ the space of bounded and locally trace class operators on $L^2(E, \mu)$.
The space $\mathscr{L}_{1, loc}(L^2(E, \mu))$ is equipped with the topology induced by the  semi-norms $ T\mapsto \| \chi_B T\chi_B\|_1$, where $\| \cdot \|_1$ is the trace class norm, $B$ ranges over bounded Borel subsets of $E$.

For any Borel subset $W \subset E$, we denote by $\mathcal{F}(W): = \sigma (\#_A: A\subset W)$
 the $\sigma$-algebra on $\Conf(E)$ generated by the mappings $\#_A: \Conf(E) \rightarrow\R$ defined by $\#_A(X) := \#(X \cap A)$, where $A$ ranges over all bounded Borel subsets of $W$.
We are now ready to formulate our key lemma.

\begin{lemma}\label{key-lem}
Let $W\subset E$ be a Borel subset, let $B_1 \subset \cdots \subset B_n \subset \cdots \subset W$ be an increasing exhausting sequence of bounded Borel subsets of $W$. For $\PP_K$-almost every $X\in \Conf(E)$  there exists a positive self-adjoint  contraction $K^{[X, \, W]} \in \mathscr{L}_{1, loc}(L^2(E\setminus W, \mu))$ such that
\[
\chi_{E\setminus W} K^{[X, \, B_n]} \chi_{E\setminus W} \xrightarrow[\text{in $\mathscr{L}_{1, loc}(L^2(E\setminus W, \mu))$}]{n\to\infty}  K^{[X, \, W]}
\]
and
\[
\PP_K(\cdot | X, W) = \PP_{K^{[X, \, W]}}.
\]
\end{lemma}

\begin{remark*}
For fixed $W$, the kernel-valued function $X \mapsto K^{[X, \, W]}$ almost surely does not depend on the choice of the approximating sequence $B_1 \subset \cdots \subset B_n \subset \cdots \subset W$.
\end{remark*}

 \subsubsection{The local property and the martingale lemma}

At the centre of our argument lies

\begin{lemma}[First local property of conditional kernels]\label{main-local}
Let $B\subset E$ be a  bounded Borel subset and  let  $Q$ be  an orthogonal projection, acting in $L^2(E, \mu)$,  such that
 $\Ran(Q) \subset L^2(E\setminus B, \mu)$ and the operator $QKQ$ is  locally trace class.
For $\PP_K$-almost every $X\in \Conf(E)$, we have
\begin{equation}\label{bqk}
\Big((Q + \chi_B)K (Q+\chi_B)\Big)^{[X, \, B]} = (Q  + \chi_B )K^{[X, \, B]} (Q  + \chi_B) =  Q K^{[X, \, B]} Q.
\end{equation}
\end{lemma}

 \begin{remark*}The formula (\ref{bqk}) is a strengthening,  on the level of kernels,  of the general property of point processes that
conditioning on the restriction to a subset commutes with the forgetting projection onto a larger subset;
see Proposition \ref{prop-local-meas} below.
 The local property  can be interpreted in terms of Neretin's formalism in \cite{neretin-fock}:
a determinantal measure is viewed as a ``determinantal state''  on a specially constructed algebra, and in order that conditional states  themselves  be determinantal the local property must take place. The local property can thus be seen as the noncommutative analogue of the fact that the operation of conditioning commutes with the operation of restriction of a configuration onto a subset.
\end{remark*}

Let $A, B$ be two disjoint bounded Borel subsets of $E$.
 It is a general property of point processes that conditioning first on $A$ and then on $B$ amounts to a  single conditioning on $A\cup B$.
A manifestation of this general property on the level of kernels is
\begin{lemma}[Second local property of conditional kernels]\label{main-local-bis}
Let $A, B$ be two disjoint bounded Borel subsets of $E$. For $\PP_K$-almost every $X\in \Conf(E)$, we have
\begin{align*}
(K^{[X, \, A]})^{[X, \, B]} = (K^{[X, \, B]})^{[X, \, A]} = K^{[X, \, A\cup B]}.
\end{align*}
\end{lemma}

Using the local properties, we establish the following  key martingale property of the kernels $K^{[X,B]}$.
\begin{lemma}\label{lem-mart-prop}
Let $W\subset E$ be a Borel subset, let $B_1 \subset \cdots \subset B_n \subset \cdots \subset W$ be an increasing exhausting sequence of bounded Borel subsets of $W$. The sequence of random variables
\[
\Big( \chi_{E\setminus W}  K^{[X, \, B_n]} \chi_{E\setminus W}  \Big)_{n\in \N}
\]
is an $(\mathcal{F}(B_n))_{n\in\N}$-adapted operator-valued martingale defined on the probability space $(\Conf(E), \mathcal{F}(E), \PP_K)$.
\end{lemma}

By definition, we have $K^{[X, \, B]} =  K^{[X \cap B, \, B]}$. Hence the mapping $X \mapsto K^{[X, \, B]}$ is an $\mathcal{F}(B)$-measurable operator-valued random variable defined on the probability space $(\Conf(E), \mathcal{F}(E), \PP_K)$.  Lemma  \ref{lem-mart-prop} is equivalent to the claim that,  for any $\varphi \in L^2(E\setminus W, \mu)$, the sequence
$
\big(  \big\langle  \chi_{E\setminus W}  K^{[X, \, B_n]} \chi_{E\setminus W} \varphi, \varphi\rangle  \big)_{n\in \N}
$
is an $(\mathcal{F}(B_n))_{n\in\N}$-adapted real-valued martingale defined on the probability space $(\Conf(E), \mathcal{F}(E), \PP_K)$.  This notion of being a martingale is equivalent to the general notion of  Frechet space valued martingales, cf. Pisier  \cite{pisier-B-martingale}.

\begin{remark*}
The proof of Lemma \ref{lem-mart-prop} below in fact yields a stronger statement:  the sequence of exterior power operators
\[
\Big( \big( \chi_{E\setminus W}  K^{[X, \, B_n]} \chi_{E\setminus W}\big)^{\wedge m}  \Big)_{n\in \N}
\]
is an $(\mathcal{F}(B_n))_{n\in\N}$-adapted operator-valued martingale, defined on the probability space $(\Conf(E), \mathcal{F}(E), \PP_K)$ and  almost surely convergent to
$
\big( \chi_{E\setminus W}  K^{[X, \, W]} \chi_{E\setminus W}\big)^{\wedge m}.
$
\end{remark*}

\noindent {\bf Acknowledgements.}
We are deeply grateful to Alexander Borichev and Alexei Klimenko for useful discussions and very helpful comments.
The research of A. Bufetov on this project has received funding from the European Research Council (ERC) under the European Union's Horizon 2020 research and innovation programme under grant agreement No 647133 (ICHAOS). It has also been funded by the Gabriel Lam\'e Chair  at the Chebyshev Laboratory of the SPbSU, a joint initiative of the French Embassy in the Russian Federation and the Saint-Petersburg State University.
 Y. Qiu is supported by the grant IDEX UNITI - ANR-11-IDEX-0002-02, financed by Programme ``Investissements d'Avenir'' of the Government of the French Republic managed by the French National Research Agency and is partially supported by NSF of China (Grant No. 11688101).
Part of this work was done at the CIRM in the framework of ``recherche en petit groupe'' programme, at Ushinsky University Yaroslavl  and the  Lomonosov Arctic University (Koryazhma branch). We are deeply grateful to these institutions for their  warm hospitality.

\section{Conditional processes and martingales}\label{sec-preliminaries}

\subsection{Martingales and the Radon-Nikodym property}\label{sec-conv-mart}
\subsubsection{Vector-valued and measure-valued martingales}

Let $(\Omega, \mathscr{F}, (\mathscr{F}_n)_{n=1}^\infty, \bold{P})$ be a filtered probability space. Let $\mathfrak{B}$ be a Banach space.  A map $F: \Omega\rightarrow \mathfrak{B}$ is called  Bochner measurable, if there exists a sequence $F_n$ of measurable, in the usual sense, step functions such that $F_n(\omega) \to F(\omega)$ almost everywhere. For any $1\le p < \infty$, we denote by $L^p(\Omega, \mathscr{F}, \bold{P}; \mathfrak{B})$ the set of all Bochner measurable functions $F: \Omega \rightarrow \mathfrak{B}$, such that $\int_\Omega \| F(\omega)\|_{\mathfrak{B}}^p \bold{P}(\dd \omega) < \infty$.  The space  $L^p(\Omega, \mathscr{F}, \bold{P}; \mathfrak{B})$ is a Banach space with the norm
 \[
 \|F\|_{L^p(\mathfrak{B})}: =  (\int_\Omega \| F(\omega)\|_{\mathfrak{B}}^p \bold{P}(\dd \omega))^{1/p}.
 \]
The algebraic tensor product $L^p(\Omega, \mathscr{F}, \bold{P}) \otimes \mathfrak{B}$ is dense in  $L^p(\Omega, \mathscr{F},\bold{P}; \mathfrak{B})$.  The operator
  \[
  \E[\cdot | \mathscr{F}_n] \otimes Id_{\mathfrak{B}} :  L^p(\Omega, \mathscr{F}, \bold{P}) \otimes \mathfrak{B} \rightarrow L^p(\Omega, \mathscr{F}, \bold{P}) \otimes \mathfrak{B}
  \]
  extends uniquely to a bounded linear operator on $L^p(\Omega, \mathscr{F}, \bold{P};  \mathfrak{B})$, for which we keep the name ``conditional expectation''  and the notation, thus obtaining the operator  $  \E[\cdot | \mathscr{F}_n]:  L^p(\Omega, \mathscr{F}, \bold{P}; \mathfrak{B}) \rightarrow L^p(\Omega, \mathscr{F}, \bold{P};\mathfrak{B}).
$
A sequence $(R_n)_{n=1}^\infty$ in $L^p(\Omega, \mathscr{F}, \bold{P}; \mathfrak{B})$ is called an $(\mathscr{F}_n)_{n=1}^\infty$-adapted martingale if $R_n = \E[R_{n+1} | \mathscr{F}_n]$ for any $n\in \N$.

    Assume  now that $\mathfrak{B}$ is a {\it separable} space. Then there exists a countable subset $D$ of the unit ball of the dual space $\mathfrak{B}^*$ such that for any $x\in \mathfrak{B}$, we have $\| x \| = \sup_{\xi \in D} | \xi (x)|$.   We will need the Pettis measurability theorem for separable Banach spaces.

   \begin{proposition}[{\cite[p. 278]{Pettis}}]\label{prop-pettis}
     A function $F: \Omega \rightarrow \mathfrak{B}$ is Bochner measurable with respect to $\mathscr{F}$ if and only if for any $\xi \in D$, the scalar function $\omega \rightarrow \xi(F(\omega))$ is $\mathscr{F}$-measurable.  A sequence $(R_n)_{n=1}^\infty$ in $L^p(\Omega, \mathscr{F}, \bold{P}; \mathfrak{B})$ is an $(\mathscr{F}_n)_{n=1}^\infty$-adapted martingale if and only if for any $\xi \in D$, the sequence $(\xi(R_n))_{n=1}^\infty$ is an $(\mathscr{F}_n)_{n=1}^\infty$-adapted martingale.
     \end{proposition}

In this paper, we  apply Proposition \ref{prop-pettis} in the particular case when $\mathfrak{B} = \mathscr{L}_1(L^2(E,\mu))$ and $D$ is the set of contractive finite rank operators on $L^2(E, \mu)$. Martingales in $\mathscr{L}_{1, loc}(L^2(E, \mu))$ are reduced to the previous case by restricting onto $L^2(B, \mu)$ with $B$ a bounded Borel subset of $E$.

Let $(T, \mathscr{A})$ be topological space equipped with the $\sigma$-algebra of Borel subsets of $T$.  We denote by $\mathfrak{P}(T, \mathscr{A})$ the set of probability measures on $(T, \mathscr{A})$.   A map $M: \Omega \rightarrow \mathfrak{P}(T, \mathscr{A})$ is called a random probability measure if for any $A\in\mathscr{A}$, the map $\omega \mapsto M(\omega, A): = M(\omega)(A)$ is measurable. A sequence of random probability measures $(M_n)_{n=1}^\infty$  is called an $(\mathscr{F}_n)_{n=1}^\infty$-adapted measure-valued martingale  on $(T, \mathscr{A})$ if for any $A \in \mathscr{A}$, the sequence $(M_n(\cdot, A))_{n\in\N}$ is a usual $(\mathscr{F}_n)_{n=1}^\infty$-adapted martingale.

\subsubsection{The Radon-Nikodym property}

In proving convergence of conditional kernels, we will use the Radon-Nikodym property for the space of trace class operators. Here we briefly recall the Radon-Nikodym property for Banach spaces; see Dunford-Pettis  \cite{RNP-1},  Phillips \cite{ RNP-Phillips} and Chapter 2 in Pisier's recent monograph \cite{pisier-B-martingale} for a more detailed exposition.

Let $\mathfrak{B}$ be a Banach space. Let $(\Omega, \mathscr{F})$ be a measurable space. Any $\sigma$-additive map $m: \mathscr{F} \rightarrow \mathfrak{B}$ is called a ($\mathfrak{B}$-valued) vector measure. A vector measure $m$ is said to have finite total variation if
\[
\sup \Big\{\sum_{i=1}^n \| m(A_i) \|_{\mathfrak{B}}  \big|\text{$\Omega = \bigsqcup_{i=1}^n  A_i$ is a measurable partition of $\Omega$} \Big\}< \infty.
\]
Given a probability measure $\bold{P}$ on $(\Omega, \mathscr{F})$, we say that the vector measure $m$ is absolutely continuous with respect to $\bold{P}$ if there exists a non-negative  function $w\in L^1(\Omega, \mathscr{F}, \bold{P})$ such that
\[
\| m(A)\|_{\mathfrak{B}} \le \int_{A} w \dd \bold{P}  \quad \text{for any $A\in \mathscr{F}$.}
\]

\begin{definition}
A Banach space $\mathfrak{B}$  is said to have the Radon-Nikodym property if for any probability space $(\Omega, \mathscr{F}, \bold{P})$ and any $\mathfrak{B}$-valued measure $m$ on $(\Omega, \mathscr{F})$, with $m$ having finite total variation and being absolutely continuous with respect to $\bold{P}$, there exists a Bochner integrable function $F_m\in L^1(\Omega, \mathscr{F}, \bold{P}; \mathfrak{B})$ such that
\[
 m(A) =  \int_{A}  F_m \dd\bold{P}  \quad \text{for any $A\in \mathscr{F}$.}
\]
\end{definition}
 By Theorem 2.9 in Pisier \cite{pisier-B-martingale}, the Radon-Nikodym property is equivalent to either of
the two requirements
\begin{enumerate}
\item Every $\mathfrak{B}$-valued martingale bounded in $L^1(\mathfrak{B})$ converges almost surely;
\item Every uniformly integrable $\mathfrak{B}$-valued martingale bounded in $L^1(\mathfrak{B})$ converges almost surely and in $L^1(\mathfrak{B})$.
\end{enumerate}

Corollary 2.15 in  Pisier \cite{pisier-B-martingale} states that if $\mathfrak{B}$ is separable and  is a dual space of another Banach space, then $\mathfrak{B}$ has the Radon-Nikodym property.
The separable space $\mathscr{L}_1(L^2(E, \mu))$  of trace class operators on $L^2(E, \mu)$  is the dual space of the space of compact operators on $L^2(E, \mu)$, and we have
\begin{proposition}\label{prop-rnp}
The space  $\mathscr{L}_1(L^2(E, \mu))$ has the Radon-Nikodym property.
\end{proposition}

\subsection{Conditional measures of point processes}\label{sec-conditional}
Let $E$ be a locally compact   $\sigma$-compact Polish space, endowed with a positive  $\sigma$-finite Radon measure $\mu$. We assume that the metric on $E$ is such that  any bounded set  is relatively compact, see Hocking and Young \cite[Theorem 2-61]{Hocking-Young}.

 A configuration $X = \{x_i\}$ on $E$ is by definition a {\it locally finite} countable subset of $E$, possibly with multiplicities.   A configuration is called simple if all points in it have multiplicity one. Let $\Conf(E)$ denote  the set of all configurations on $E$. The mapping $X \mapsto N_X: = \sum_{i} \delta_{x_i}$ embeds $\Conf(E)$ into the space of Radon measures on $E$.  Under the {\it vague topology}, $\Conf(E)$  is a Polish space, see, e.g., Daley and Vere-Jones \cite[Theorem 9.1. IV]{Daley-Vere}.  By definition, a {\it point process} on $E$ is a Borel probability  measure $\PP$ on $\Conf(E)$. We call $\PP$ simple if $\PP(\{X: \text{$X$ is simple}\}) = 1$.

For a Borel subset $W \subset E$, let $\mathcal{F}(W)$ be the $\sigma$-algebra on $\Conf(E)$ generated by all mappings $X \mapsto \#_B(X) : = \#(X \cap B)$, where $B \subset W$ are bounded Borel subsets; the algebra $\mathcal{F}(E)$ coincides with the Borel $\sigma$-algebra on $\Conf(E)$.

Take a Borel subset $W\subset E$. A Borel probability measure  $\PP$ on  $\Conf(E)$ can be  viewed as a measure on $\Conf(W) \times \Conf(W^c)$; we shall sometimes write  $\PP=\PP_{W, W^c}$ to stress dependence on $W$.

Denote by  $(\pi_W)_{*}(\PP)$  the image measure of $\PP$ under the surjective mapping $ \pi_W:  \Conf(E)  \rightarrow \Conf(W)$ defined by $\pi_W(X) = X\cap W$.  By  disintegrating the probability measure $\PP_{W, W^c}$, for $(\pi_W)_{*}(\PP)$-almost every configuration $X_{0}\in \Conf(W)$,  there exists a probability measure, denoted by $\PP(\cdot| X_{0}, W)$,  supported on $ \{X_{0}\} \times \Conf(W^c) \subset \Conf (E) $,  such that
\[
\PP_{W, W^c}= \int_{\Conf(W)}  \PP(\cdot| X_{0}, W)   (\pi_W)_{*}(\PP) (  \dd  X_{0}).
\]
The measure $\PP(\cdot| X_{0}, W)$ is referred to as the  conditional measure on $\Conf(W^c)$ or conditional point process on $W^c$ of $\PP$, the condition being that the configuration on $W$ coincides with $X_0$. In what follows, we denote also
\[
\PP (\cdot | X, W): = \PP(\cdot |  X\cap W, W), \quad \text{for $\PP$-almost every configuration $X\in \Conf(E)$}.
\]
Moreover, for a random variable $f \in L^1( \Conf(E), \PP)$, we will denote by
\[
\E_{\PP}(f| X, W) :=  \E_{\PP} [f |  \mathcal{F}(W)](X\cap W).
\]

\begin{proposition}\label{prop-local-meas}
Let $W_1, W_2$ be two disjoint Borel subsets of $E$.  For $\PP$-almost every $X\in \Conf(E)$, we have
\begin{align}\label{local-general-bis}
(\pi_{W_1\cup W_2})_{*}[\PP] (\cdot | X, W_1)  = ( \pi_{W_1\cup W_2})_{*} [\PP(\cdot| X, W_1)].
\end{align}

\end{proposition}
\begin{proof}
First we have
\begin{align*}
\PP = \int_{\Conf(E)}       \PP(\cdot| X, W_1) \PP(\dd X) \an (\pi_{W_1\cup W_2})_{*} [\PP] = \int_{\Conf(E)}    (\pi_{W_1\cup W_2})_{*}  [  \PP(\cdot| X, W_1) ] \PP(\dd X).
\end{align*}
Since $\PP(\cdot | X, W_1)$ is supported on the subset $\{Y \in \Conf(E):   Y \cap W_1 = X \cap W_1\}$, and $(\pi_{W_1\cup W_2} )_{*} [  \PP(\cdot| X, W_1) ]$ is supported on $\{Z\in \Conf(C \cup B):  Z \cap B = X \cap B\}$, by the uniqueness of conditional measures, we get \eqref{local-general-bis}.
\end{proof}
Since $\PP(\cdot| X, W)$ is by definition supported on $\{X\cap W\} \times \Conf(W^c)$, we consider $\PP(\cdot| X, W)$ as a measure on $\Conf(W^c)$. Further identifying $\Conf(W^c)$ with the subset
$
\Conf(E, W^c): =  \{X \in \Conf(E): X \cap W = \emptyset\}
 \subset \Conf(E),
$
when it is necessary, we may also view $\PP(\cdot| X, W)$ as a measure on $\Conf(E)$ supported  on the subset $\Conf(E, W^c)$.

\subsection{Palm measures}\label{sec-palm-cond}

The $n$-th correlation measure $\rho_{n, \PP}$ of a point process $\PP$ on $E$, if it exists, is the unique $\sigma$-finite Borel measure on $E^n$ satisfying
\[
\rho_{n, \PP} (A_1^{k_1} \times  \cdots \times A_j^{k_j})  =  \int_{\Conf(E)}   \prod_{i =1}^j  \frac{ \#(X \cap A_i)!}{( \# (X \cap A_i)  - k_i)! } \dd\PP( X),
\]
for all bounded disjoint Borel subsets $A_1, \cdots, A_j \subset E$ and positive integers $k_1, \cdots, k_j$ with $k_1 + \cdots + k_j = n$.  Here if $ \#(X \cap A_i)<k_i$, we set $\#(X \cap A_i)!/( \#(X \cap A_i) - k_i)! = 0$.

For example, the $n$-th correlation measure of a determinantal process $\PP_K$ is given by
\[
\rho_{n, \PP_{K}}  (\dd x_1 \cdots \dd x_n) = \det(K(x_i, x_j))_{1\le i, j \le n}  \cdot \mu^{\otimes n} (\dd x_1 \cdots \dd x_n),
\]
where $K(x, y)$ is the integral kernel of the operator $K$ satisfying \eqref{kernel-ass}.

Assume that $\PP$ is a simple point process on $E$ such that $\rho_{n, \PP}$ exists for any $n\in\N$. The reduced $n$-th order {\it Campbell measure} $\mathscr{C}_{n, \PP}^{!}$ of $\PP$ is a $\sigma$-finite measure on $E^n \times \Conf(E)$ satisfying
\begin{align*}
\int_{E^n \times \Conf(E)}   F(x, X)  \mathscr{C}_{n, \PP}^{!} (\dd x \times \dd X) =  \int_{\Conf(E)}  \Big[{\sum_{x \in X^n}}\!\!\!^{\#}   F(x,  X\setminus \{x_1, \cdots, x_n\}) \Big] \PP(\dd X),
\end{align*}
for any Borel function $F: E^n \times \Conf(E)\rightarrow\R^{+}$.  Here  ${\sum\!^{\#}}$ is the summation over all {\it ordered $n$-tuples} $(x_1, \cdots, x_n)$ with distinct coordinates $x_1, \cdots, x_n \in X$. Disintegrating $\mathscr{C}_{n, \PP}^{!} (\dd x \times \dd X) $, we obtain
 \begin{align}\label{def-Palm}
\int_{E^n \times \Conf(E)}   F(x, X)  \mathscr{C}_{n, \PP}^{!} (\dd x \times \dd X)  = \int_{E^n} \rho_{n, \PP}(\dd x) \int_{\Conf(E)}   F(x, X)   \PP^{x} ( \dd X),
\end{align}
where the probability measures $\PP^x$ are defined for $\rho_{n, \PP}$-almost every $x\in E^n$ and are called {\it reduced} Palm measures of $\PP$. In what follows, by Palm measures we always mean reduced Palm measures. Since $\PP^{x_1, \cdots, x_n}$ is invariant under permutation of the coordinates in $(x_1, \cdots, x_n)$, we may write
\[
\PP^{X} := \PP^{x_1, \cdots, x_n}, \quad \text{if $X  = \{x_1, \cdots, x_n \}$.}
\]

\subsection{Determinantal point processes, conditioning on bounded subsets}\label{sec-DPP-cond}
Let $W\subset E$ be a Borel subset. Recall that, by definition, the push-forward $(\pi_W)_{*}(\PP_K)$ is a determinantal point process on $W$, induced by a correlation kernel $\chi_{W} K \chi_{W}$.
We next recall, for  determinantal point processes, the form of conditional measures with respect to restricting the configuration on a bounded subset.

Recall that $\Conf(W)$ is identified as a subset $\{X \in \Conf(E): X \subset W\}$ of $\Conf(E)$.  Given a point process $\PP$ on $E$, that is, a Borel probability on $\Conf(E)$, we set
\begin{align}\label{nor-res}
\overline{\PP\restriction_{\Conf(W)}}: = \left\{ \begin{array}{cc} \displaystyle \frac{\PP\restriction_{\Conf(W)}}{\PP(\Conf(W))}, &\text{ if $\PP(\Conf(W)) >0$}
\\
0, &\text{ if $\PP(\Conf(W)) =0$}
\end{array}
\right..
\end{align}
Let $B\subset E$ be a bounded Borel subset.   If $\PP_K(\Conf(B^c))>0$, then, by \cite[Proposition 2.1]{Buf-inf},
 $\overline{\PP_K\restriction_{\Conf(B^c)}}$ is a determinantal point process on $B^c$ induced by the correlation kernel $\chi_{B^c} K ( 1 - \chi_B K)^{-1}  \chi_{B^c}$; in the discrete setting, see also Borodin and Rains \cite{BorRains}, Lyons \cite{DPP-L}. The reader is also referred to \cite{BQ-JFA} for conditional measures of generalized Ginibre point processes.  Next,  By a Theorem of Shirai and Takahashi \cite[Theorem 1.7]{ST-palm},  for $\PP_K$-almost every $X \in \Conf(E)$, the Palm measure $\PP_K^{X\cap B}$ is a determinantal point process on $E$,  induced by the  correlation kernel
\[
K^{X\cap B} = K^{p_1, \cdots, p_n}, \quad \text{if $X \cap B = \{p_1, \cdots, p_n\}$} ;
\]
Summing up, we obtain
\begin{proposition}\label{prop-DPP}
  $\PP_K(\cdot|X, B)$ is a determinantal point process on $B^c$  for $\PP_K$-almost every $X \in \Conf(E)$,  induced by a correlation kernel $K^{[X, \, B]}$ defined in \eqref{def-cond-kernel}.
\end{proposition}

\begin{proof}
Indeed, by Proposition \ref{prop-bdd-cond} in the Appendix  below,  for $\PP_K$-almost every $X \in \Conf(E)$, we have
\[
 \PP_K(\cdot | X, B) = \overline{\PP_K^{X\cap B}\restriction_{\Conf(B^c)}} = \PP_{K^{[X, \, B]}}.
 \]
\end{proof}

\begin{proposition}
If  $K$ is the orthogonal projection onto a closed subspace $H \subset L^2(E, \mu)$, then  the kernel $K^{p_1, \cdots, p_n}$ corresponds to the orthogonal projection from $L^2(E, \mu)$ onto the subspace
\[
H(p_1, \cdots, p_n):  = \{h \in H: h(p_1) = \cdots = h(p_n) = 0 \}.
\]
Moreover, for a bounded Borel subset $B\subset E$, the operator $K^{[X, \, B]}$ is the orthogonal projection onto the closure of the subspace
\[
\chi_{E\setminus B} H(X \cap B) = \{\chi_{E\setminus B} h: h \in H(X\cap B)\}.
\]
\end{proposition}

\begin{proof}
The first assertion can be proved by induction on $n$, by noting that $ K^{p_1, \cdots, p_n} = ((K^{p_1})^{\cdots})^{p_n}$.  In particular, when $n = 1$, the equality $K^{p_1}(x, y) = K (x, y)  -\frac{K(x, p_1) K(p_1, y)}{ K(p_1, p_1)}$ implies that $K^{p_1} = K - \Pi_{K_{p_1}}$ where $\Pi_{K_{p_1}}$ is the one-rank orthogonal projection onto the linear space spanned by the function $K_{p_1}(\cdot) = K(\cdot, p_1)$. Therefore, $K^{p_1}$ is the orthogonal projection onto $H(p_1)$.

 The second assertion is an immediate consequence of \cite[Proposition 2.5]{BQS}
\end{proof}

\section{The local property: proof of Lemmata \ref{main-local}, \ref{main-local-bis}.}

\subsection{Proof of Lemma \ref{main-local}.}

Let $B\subset E$ be a  bounded Borel subset and  let  $Q: L^2(E, \mu) \rightarrow L^2(E,\mu)$ be  an orthogonal projection whose range satisfies $\Ran(Q) \subset L^2(E\setminus B, \mu)$ and such that $QKQ$ is
locally trace-class.  Introduce a positive contractive locally trace-class operator $R$ by the formula
\begin{align}\label{def-R}
R = R(K, B, Q) : =(Q+\chi_B)K(Q+\chi_B).
\end{align}

Recall that from the introduction, we fixed a Borel subset $E_0 \subset E$, such that  $\mu(E \setminus E_0) =0$ and the kernel $K(x, y)$ is well-defined on $E_0\times E_0$.  Recall also the notation introduced in Definition \ref{def-canonical-kernel}.
 \begin{lemma}\label{lem-palm}
Let $R$ be the operator introduced in \eqref{def-R}.  For any $p\in B \cap E_0$, we have $R^p= (Q+\chi_B)  K^p  (Q+\chi_B)$.
More generally, for  $(p_1, \cdots, p_n)\in (B\cap E_0)^n$, we have
\[
R^{p_1, \cdots, p_n}= (Q+\chi_B)  K^{p_1, \cdots, p_n}  (Q+\chi_B) .
\]
In particular,
\[
R^{X\cap B} = (Q + \chi_B) K^{X\cap B} (Q +\chi_B),\quad \text{for $\PP_K$-almost every $X\in \Conf(E)$.}
\]
\end{lemma}

\begin{proof}
Take an orthonormal basis $\varphi_i$ of the range $\Ran(Q) \subset L^2(E\setminus B, \mu)$ of $Q$ and write
\[
Q = \sum_{i\in \N} \varphi_i \otimes \overline{\varphi_i}.
\]
We may assume that the values $\varphi_i(x)$ are well-defined for any index $i \in \N$ and any $x\in E_0$.  Observe that for any $p\in B\cap E_0$, we have
\begin{align}\label{R-vector}
R(\cdot, p) =  (Q + \chi_B) [K(\cdot, p)].
\end{align}
Indeed, write
\[
R = (\sum_{i\in \N}\varphi_{i} \otimes \overline{\varphi_i} ) K (\sum_{j\in \N} \varphi_j \otimes \overline{\varphi_j}) + (\sum_{i\in \N}\varphi_{i} \otimes \overline{\varphi_i} ) K \chi_B + \chi_B K (\sum_{j\in \N} \varphi_j \otimes \overline{\varphi_j})  +\chi_B K \chi_B,
\]
since $p\in B\cap E_0$, we get for any $x \in E_0$:
\begin{align*}
R(x, p)  &=   \sum_{i\in \N}\varphi_{i}(x) \int_{E}\overline{\varphi_i(y)}  K(y, p)  \mu(\dd y)+\chi_B(x)  K(x, p)
\\
& = \sum_{i\in \N} \varphi_i(x) \langle  K(\cdot, p), \varphi_i \rangle + \chi_B(x)  K(x, p),
\end{align*}
which is equivalent to  \eqref{R-vector}.  Since $R(p, p) =  K(p, p)$, we have
\begin{align*}
R^p & = R  - \frac{R(\cdot, p) \otimes \overline{R(\cdot, p)}}{R(p, p)}  = (Q+\chi_B)K(Q+\chi_B) -    \frac{(Q + \chi_B) [K(\cdot, p)] \otimes \overline{(Q + \chi_B) [K(\cdot, p)] }}{ K(p, p)}
\\
& = (Q+\chi_B) \Big[ K - \frac{K(\cdot, p) \otimes \overline{K(\cdot, p)}}{K(p,p)}\Big]   (Q+\chi_B)  =  (Q+\chi_B)  K^p  (Q+\chi_B).
\end{align*}
The formula for $R^{p_1, \cdots, p_n}$ follows immediately by  induction on $n$.
\end{proof}

Recall that, by our discussion in \S \ref{sec-DPP-cond}, the  kernel
$
\chi_{E\setminus B}K(1-\chi_B K)^{-1}\chi_{E\setminus B}
$
is a correlation kernel for the determinantal point process $\overline{\PP_K\restriction_{\Conf(B^c)}}$, provided that $\PP_K(\Conf(B^c))>0$.

\begin{lemma}\label{lem-res}
Let $B$ be a bounded Borel subset of $E$ such that $\PP_K(\#_B = 0) >0$. Let $R$ be the operator introduced in \eqref{def-R}.  Then
\begin{align}\label{R-K}
\chi_{E\setminus B}R(1-\chi_B R)^{-1}\chi_{E\setminus B}=Q\left(\chi_{E\setminus B}K(1-\chi_B K)^{-1}\chi_{E\setminus B}\right)Q.
\end{align}
\end{lemma}

\begin{proof}
The gap probability $\PP_K ( \#_B= 0)$ is given by
\begin{align}\label{gap-pb}
\PP_K ( \#_B= 0) = \PP_K(\{X: X\cap B = \emptyset\})  = \det (1 - \chi_B K \chi_B) >0.
\end{align}
It follows that $1 - \chi_B K \chi_B$ is invertible and hence $1$ is not an eigenvalue of $ \chi_B K \chi_B$. But since $\chi_B K \chi_B$ is a priori a positive contraction and $\chi_B K \chi_B$ is compact, its norm coincides with its maximal eigenvalue. Hence $\chi_B K \chi_B$ is strictly contractive. But we also have
\begin{align}\label{both-side-c}
\|  \chi_B K \chi_B \|  = \| ( \chi_B K^{1/2}  )  (  \chi_B K^{1/2})^* \| =  \|\chi_B K^{1/2}\|^2  < 1.
\end{align}
Hence
\begin{align}\label{one-side-c}
\| \chi_B K\| \le \| \chi_B K^{1/2}\| \|K^{1/2}\| < 1.
\end{align}
Therefore, both $\chi_B K$  and $ \chi_B R = \chi_BK(Q+\chi_B)$ are strictly contractive.  In particular, the operators on both the left hand side and the right hand side of \eqref{R-K} are well-defined.

Since $Q$ commutes with $\chi_{E\setminus B}$,  we have
\begin{align*}
\chi_{E\setminus B} R \chi_{E \setminus B}   = Q  \chi_{E\setminus B} K    \chi_{E\setminus B} Q \an \chi_{E\setminus B} R \chi_B  = Q \chi_{E\setminus B} K \chi_B.
\end{align*}
Since $\chi_B R \chi_B = \chi_B K \chi_B$, for  $n\ge 1$, we have
\begin{multline}
 \chi_{E\setminus B} R (\chi_B R)^n  \chi_{E\setminus B}
=  \chi_{E\setminus B} R  (\chi_B R) \cdots (\chi_B R)  \chi_{E\setminus B}
 =  \chi_{E\setminus B} R \chi_B  (\chi_B R \chi_B)^{n-1} \chi_B R\chi_{E\setminus B}
 =\\=Q \chi_{E\setminus B} K  \chi_B  (\chi_B K \chi_B)^{n-1} \chi_B K \chi_{E\setminus B}  Q
 = Q \chi_{E\setminus B} K (\chi_BK)^n \chi_{E\setminus B} Q.
\end{multline}
Now since $\chi_B R$ and $\chi_B K$ are both strictly contractive,  we finally write
\begin{multline}
\chi_{E\setminus B}R ( 1 - \chi_B R)^{-1} \chi_{E\setminus B} =  \sum_{n=0}^\infty \chi_{E\setminus B} R (\chi_B R)^n \chi_{E\setminus B}=
\\= \sum_{n=0}^\infty  Q \chi_{E\setminus B} K (\chi_B K)^n \chi_{E\setminus B}Q =Q  \chi_{E\setminus B}K ( 1 - \chi_B K)^{-1} \chi_{E\setminus B} Q.
\end{multline}
\end{proof}

\begin{proof}[ Conclusion of the proof of Lemma \ref{main-local}]
By Proposition \ref{prop-bdd-cond} and Proposition \ref{prop-DPP},
\[
\PP_K(\cdot | X, B) = \overline{(\PP_K)^{X\cap B}\restriction_{\Conf(B^c)}} =  \overline{\PP_{K^{X\cap B}}\restriction_{\Conf(B^c)}}, \quad \text{for $\PP$-almost every $X \in \Conf(E)$}.
\]
By definition \eqref{nor-res} of the normalized restriction measure $\overline{\PP_{K^{X\cap B}}\restriction_{\Conf(B^c)}}$, we must have
\begin{align}\label{non-vanish-proba}
\PP_{K^{X\cap B}} (\#_B = 0)  = \PP_{K^{X\cap B}} ( \Conf(B^c)) >0, \quad \text{for $\PP_K$-almost every $X\in \Conf(E)$.}
\end{align}
Lemma \ref{lem-res} applied to the operators $K^{X\cap B}$ and $R^{X\cap B}$ and Lemma \ref{lem-palm} now imply Lemma \ref{main-local}.
\end{proof}

\subsection{Proof of Lemma \ref{main-local-bis}}
 Choose an arbitrary unit vector $\varphi \in L^2(E \setminus (A \cup B), \mu)$, let $Q$ be the orthogonal projection from $L^2(E, \mu)$ onto the one dimensional subspace spanned by $\varphi$. Define
 \[
 R = R_\varphi: = (\chi_A + \chi_B + Q) K (\chi_A + \chi_B + Q).
 \]
 Arguing as in the proof of Lemma \ref{main-local}, we obtain the $\PP_K$-almost sure equalities
 \begin{align}\label{as-kernel}
 \begin{split}
 R^{[X,\, A]} =  (\chi_B + Q) K^{[X,\, A]}(\chi_B + Q);  \quad  R^{[ X, \,  A \cup B ] } =   Q K^{[X,  \, A \cup B]}  Q;  \quad (R^{[X,\, A]})^{ [X,\, B]} =   Q (K^{[X,\, A]})^{ [X,\, B]}  Q.
  \end{split}
 \end{align}
We also have the following description of conditional measures:
 \begin{align*}
 \PP_R (\cdot |   X,  A ) = \PP_{ R^{[X, \, A]}} \an
 \PP_R (\cdot |   X, A\cup B ) = \PP_{ R^{[X, \,  A \cup B ]}}, \quad \text{for $\PP_R$-almost every $X\in \Conf(E)$.}
 \end{align*}
The above first equality implies that
 \begin{align*}
 \big[ \PP_R (\cdot |   X, A )\big](\cdot | X, B ) = \PP_{ R^{[X,\, A]}}(\cdot | X, B )  = \PP_{(R^{[X,\, A]})^{ [X,\, B]}}, \quad \text{for $\PP_R$-almost every $X\in \Conf(E)$.}
 \end{align*}
Now we may apply the measure-theoretic identity
\begin{align*}
 \big[ \PP_R (\cdot |    X, A )\big](\cdot |X, B )  = \PP_R (\cdot |   X, A\cup B ), \quad \text{for $\PP_R$-almost every $X\in \Conf(E)$}
\end{align*}
and obtain
 \begin{align}\label{identity-prob}
  \PP_{ R^{[X, \,  A \cup B ]}} = \PP_{(R^{[X,\, A]})^{ [X,\, B]}}, \quad \text{for $\PP_R$-almost every $X\in \Conf(E)$.}
 \end{align}
It follows that for $\PP_R$-almost every $X\in \Conf(E)$, we have
 \begin{align*}
 \E_{\PP_R} \big[  \# ( X \cap (E \setminus (A \cup B)) \big|  X, A \cup B\big]  = \tr \Big( \chi_{E \setminus (A\cup B)} R^{[X, \, A \cup B ]}  \chi_{E \setminus (A\cup B)} \Big) = \tr \Big( \chi_{E \setminus (A\cup B)}   (R^{[X,\, A]})^{ [X,\, B]}   \chi_{E \setminus (A\cup B)} \Big).
 \end{align*}
 Combining with \eqref{as-kernel}, we obtain the $\PP_R$-almost sure equality
 \begin{align*}
 \tr \Big( \chi_{E \setminus (A\cup B)} Q K^{[ X, \, A \cup B ]}  Q \chi_{E \setminus (A\cup B)} \Big) = \tr \Big( \chi_{E \setminus (A\cup B)}   Q(K^{[X,\, A]})^{ [X,\, B]}   Q \chi_{E \setminus (A\cup B)} \Big).
  \end{align*}
That is,
 \begin{align*}
  \langle K^{ [X, \, A \cup B]} \varphi, \varphi\rangle
 =\langle  (K^{[X,\, A]})^{ [X,\, B]} \varphi, \varphi \rangle, \quad \text{for $\PP_R$-almost every $X\in \Conf(E)$.}
 \end{align*}
 Since $\varphi$ is arbitrary and since $L^2(E\setminus (A \cup B), \mu)$ is separable and both $K^{ [X, \, A \cup B]}$ and $(K^{[X,\, A]})^{ [X,\, B]}$ are supported on $L^2(E\setminus (A \cup B), \mu)$, we obtain
 \begin{align}\label{2-sets-bis}
 K^{ [X, \, A \cup B]}  =   (K^{[X,\, A]})^{ [X,\, B]}, \quad \text{for $\PP_R$-almost every $X\in \Conf(E)$.}
 \end{align}
Observe that the equality $\chi_{A\cup B} R \chi_{A\cup B} = \chi_{A\cup B} K \chi_{A\cup B}$ implies the equality $(\pi_{A\cup B})_{*} (\PP_R) = (\pi_{A\cup B})_{*} (\PP_K)$. Combining with \eqref{2-sets-bis} and the fact that $K^{ [X, \, A \cup B]}$ and $(K^{[X,\, A]})^{ [X,\, B]}$ are $\mathcal{F}( A\cup B)$-measurable, we get the desired equality
\begin{align*}
 K^{ [X, \, A \cup B]}  =   (K^{[X,\, A]})^{ [X,\, B]}, \quad \text{for $\PP_K$-almost every $X\in \Conf(E)$.}
 \end{align*}
\qed

\section{The martingale property: proof of Lemma \ref{lem-mart-prop}.}
\begin{proposition}\label{prop-1-step}
For any bounded Borel subset $B\subset E$, we have
\begin{align}\label{cond-to-empty}
\E_{\PP_K} ( K^{[X,\, B]}) =  \int_{\Conf(E)} K^{[X,\, B]}  \PP_K(\dd X) = \chi_{E\setminus B} K \chi_{E\setminus B}.
\end{align}
\end{proposition}

\begin{remark*}
Extending the argument of Benjamini, Lyons, Peres and Schramm \cite{BLPS} for the case of spanning trees, Lyons \cite[Lemma 7.17]{DPP-L} proved \eqref{cond-to-empty} when $E$ is discrete and $K$ is an orthogonal projection on $\ell^2(E)$. Our proof, based on the local property, is quite different and works both in the continuous and the discrete  setting.
\end{remark*}

\begin{proof}[Proof of Lemma \ref{lem-mart-prop} assuming Proposition \ref{prop-1-step}.]
Applying Proposition \ref{prop-1-step} to the kernel $K^{[X, \, B_n]}$  and the bounded Borel subset $B_{n+1}\setminus B_n \subset E\setminus B_n$, we obtain
\begin{align*}
\E_{\PP_{K^{[X, \, B_n]}}}\Big[  (K^{[X,\, B_n]})^{[X, \,   B_{n+1}\setminus  B_n]}  \Big] =\chi_{E\setminus B_{n+1}} K^{[X,\, B_n]} \chi_{E\setminus B_{n+1}}, \quad \text{for $\PP_{K}$-almost every $X$}.
\end{align*}
The equality  $\PP_{K^{[X, \, B_n]}} = \PP_{K} (\cdot |X,  B_n)$ now yields
\begin{align*}
 \E_{\PP_{K^{[X, \, B_n]}}}\Big[  (K^{[X,\, B_n]})^{[X, \,   B_{n+1}\setminus  B_n]}  \Big] = \E_{\PP_K}\Big[    (K^{[X,\, B_n]})^{[X, \,   B_{n+1}\setminus  B_n]}   \Big| \mathcal{F}(B_n)\Big],\quad \text{for $\PP_K$-almost every $X$.}
\end{align*}
Combining with Lemma  \ref{main-local-bis}, we get
\begin{align*}
  \E_{\PP_K}\Big[   K^{[X, \, B_{n+1}]} \Big| \mathcal{F}(B_n) \Big]  = \chi_{E\setminus B_{n+1}} K^{[X, \, B_n]} \chi_{E\setminus B_{n+1}}, \quad \text{for $\PP_K$-almost every $X$.}
\end{align*}
By linearity of the composition on the left and on the right with the operator of  multiplication by $\chi_{E\setminus W}$ and the elementary equalities $\chi_{E\setminus W} \cdot  \chi_{E\setminus B_{n+1}}  =  \chi_{E\setminus W}$,  we get the desired martingale property:
\begin{align*}
   \E\Big[  \chi_{E\setminus W} K^{[X, \, B_{n+1}]}  \chi_{E\setminus W} \Big| \mathcal{F}( B_n)\Big]   =  \chi_{E\setminus W}  K^{[X, \, B_n]} \chi_{E\setminus W}, \quad \text{for $\PP_K$-almost every $X$.}
\end{align*}
\end{proof}

\begin{proof}[Proof of Proposition \ref{prop-1-step}]
Let $\varphi  \in L^2(E\setminus B, \mu)$ be such that $\| \varphi\|_2=1$. We use \eqref{def-R} for $Q = \varphi \otimes \overline{\varphi}$, the orthogonal projection onto the one-dimensional space spanned by $\varphi$, and thus set
\[
R = (\varphi \otimes \overline{\varphi}+\chi_B)K(\varphi \otimes \overline{\varphi}+\chi_B).
\]
 We have the clear identity
\begin{equation}\label{pibkr}
(\pi_B)_*(\PP_R) = \PP_{\chi_B R\chi_B}=  \PP_{\chi_B K \chi_B} = (\pi_B)_*(\PP_K).
\end{equation}
By Lemma \ref{main-local}, for $\PP_K$-almost every $X\in \Conf(E)$, we have
\begin{align*}
R^{[X, \, B]} =  Q K^{[X, \, B]} Q  = (\varphi \otimes \overline{\varphi}) K^{[X, \, B]} (\varphi \otimes \overline{\varphi}).
\end{align*}
Since clearly $K^{[X, \, B]} = K^{[X \cap B, \, B]}$ and $R^{[X, \, B]} = R^{[X \cap B, \, B]}$, the above equality holds  for $\PP_R$-almost every $X\in \Conf(E)$.
Now recall  that $\PP_R (\cdot| X, B) = \PP_{R^{[X,\, B]}}$, for $\PP_R$-almost every $X\in \Conf(E)$. Hence
\begin{align*}
\E_{\PP_{R}}\big[\#_{E\setminus B} \big| X, B \big]= \E_{\PP_{R^{[X, \, B]}}}\big[\#_{E\setminus B} \big] =  \tr (\chi_{E\setminus B}R^{[X, \, B]}  \chi_{E\setminus B})  =\langle  K^{[X, \, B]} \varphi, \varphi \rangle, \, \text{for $\PP_R$-almost every $X\in \Conf(E)$}.
\end{align*}
Consequently,
\begin{align*}
\E_{\PP_{R}}\big[\#_{E\setminus B} \big] =\tr(\chi_{E\setminus B} R \chi_{E\setminus B})  = \tr(Q  K    Q) = \langle K \varphi, \varphi\rangle.
\end{align*}
On the other hand,
\begin{align*}
\E_{\PP_{R}}\big[\#_{E\setminus B} \big] = \E_{\PP_R} \Big(\E_{\PP_{R}}\big[\#_{E\setminus B} \big| X, B \big]  \Big)  = \E_{\PP_R} \Big( \langle  K^{[X, \, B]} \varphi, \varphi \rangle  \Big),
\end{align*}
whence
\[
  \E_{\PP_R} \Big( \langle  K^{[X, \, B]} \varphi, \varphi \rangle  \Big) = \langle K \varphi, \varphi\rangle.
\]
 The relation  $K^{[X, \, B]} = K^{[X\cap B, \, B]}$ and the identity (\ref{pibkr}) together give
\begin{align*}
 \E_{\PP_R} \Big( \langle  K^{[X, \, B]} \varphi, \varphi \rangle  \Big) &=   \E_{\PP_R} \Big( \langle  K^{[X\cap B, \, B]} \varphi, \varphi \rangle  \Big) = \E_{\PP_K} \Big( \langle  K^{[X\cap B, \, B]} \varphi, \varphi \rangle  \Big)  =  \E_{\PP_K} \Big( \langle  K^{[X, \, B]} \varphi, \varphi \rangle  \Big)
\end{align*}
and
\begin{equation}\label{exp-kxb}
  \E_{\PP_K} \Big( \langle  K^{[X, \, B]} \varphi, \varphi \rangle  \Big) = \langle K \varphi, \varphi\rangle.
\end{equation}
Since $\varphi$ is an arbitrarily chosen unit function in $L^2(E\setminus B)$ and since $K^{[X, \,B]} = \chi_{E\setminus B} K^{[X, \,B]} \chi_{E\setminus B}$, we obtain \eqref{cond-to-empty}.
\end{proof}

\section{Proof of Lemma \ref{key-lem}}

\begin{proposition}\label{prop-RNP}
Let $W\subset E$ be a Borel subset, and let $B_1 \subset B_2 \subset \cdots B_n\subset \cdots \subset W $ be an increasing exhausting sequence of bounded Borel subsets of  $W$.  The sequence $\big( \chi_{E\setminus W}  K^{[X, \, B_n]} \chi_{E\setminus W}  \big)_{n\in \N}$
converges $\PP_K$-almost surely in the space of locally trace class operators.
\end{proposition}

\begin{proof}
Since $K$ is locally of trace class, there exists a positive function $\psi: E\setminus W \rightarrow (0, 1]$ such that $\psi^{1/2}K \psi^{1/2} $ is of trace class and for any bounded subset $B\subset E$, we have
\begin{align}\label{inf-psi}
\inf_{x\in B} \psi(x)>0.
\end{align}
Then
\[
\E_{\PP_K}(\sum_{x\in X}\psi(x)  ) =  \int_{E}   \psi(x) K(x, x) \mu(\dd x) = \tr(\psi^{1/2} K \psi^{1/2} )= M_\psi< \infty.
\]
Denote
\[
G(X, n): = \chi_{E\setminus W}  K^{[X, \, B_n]} \chi_{E\setminus W}.
\]
Then for any $n\in\N$, we have
\begin{align}\label{bdd-mart}
M_\psi=  \E_{\PP_K}(\sum_{x\in X}\psi(x)  )   = \E_{\PP_K}\Big[\E_{\PP_K} \big(    \sum_{x\in X}\psi(x)  \big| \mathcal{F}( B_n)\big) \Big]  =  \E_{\PP_K}\big[ \tr(\psi^{1/2} G(X, n)  \psi^{1/2}  )\big].
\end{align}
By the martingale property of the sequence $(G(X, n))_{n\in\N}$ and the equality \eqref{bdd-mart}, the sequence $(\psi^{1/2} G (X, n)  \psi^{1/2})_{n\in\N}$ forms a {\it bounded} martingale in $L^1(\PP_K,  \mathscr{L}_1(L^2(E, \mu)))$.  By Proposition \ref{prop-rnp}, the Banach space $\mathscr{L}_1(L^2(E, \mu))$ has  the Radon-Nikodym property. Therefore there exists a measurable function $F (X, \infty)$ with values in  $\mathscr{L}_1(L^2(E, \mu))$, such that
\[
\psi^{1/2} G(X, n)  \psi^{1/2} \xrightarrow[\text{$\PP_K$-a.s.}]{\text{in $\mathscr{L}_1(L^2(E, \mu))$}} F(X, \infty).
\]
The assumption \eqref{inf-psi} implies that $ \psi^{-1/2} F(X, \infty)  \psi^{-1/2} \in \mathscr{L}_{1, loc}(L^2(E, \mu))$ and we have
\begin{align}\label{kernel-limit}
\chi_{E\setminus W}  K^{[X, \, B_n]} \chi_{E\setminus W} = G (X, n)  \xrightarrow[ \text{$\PP_K$-a.s.}]{\text{in $\mathscr{L}_{1, loc}(L^2(E, \mu))$}}  \psi^{-1/2} F(X, \infty)  \psi^{-1/2}.
\end{align}
\end{proof}

\begin{proof}[Proof of Lemma \ref{key-lem}]
By \eqref{cond-mart-weak}, for $\PP_K$-almost every $X\in \Conf(E)$, we have
\begin{align}\label{bdd-app}
(\pi_{W^c})_{*}[\PP_K(\cdot |X,  B_n)]   \xrightarrow[\text{weakly}]{n\to\infty} \PP_K(\cdot | X, W).
\end{align}
By items (i) and (iv) of Proposition \ref{prop-DPP}, for $\PP_K$-almost every $X\in \Conf(E)$, we have
\begin{align}\label{bdd-dpp}
(\pi_{W^c})_{*}[\PP_K(\cdot |X,  B_n)]  = \PP_{\chi_{E\setminus W}  K^{[X, \, B_n]} \chi_{E\setminus W}}.
\end{align}
Combining \eqref{kernel-limit}, \eqref{bdd-app} and \eqref{bdd-dpp} with the fact that the convergence of correlation kernels in $\mathscr{L}_{1, loc}(L^2(E, \mu))$ implies the weak convergence of the corresponding determinantal measures, we complete the proof of Lemma \ref{key-lem}.
\end{proof}
We conclude this section with a simple general proposition that allows us to construct bounded martingales from the sequence $\big(K^{[X, \, B_n]} \big)_{n\in \N}$.

\begin{proposition}\label{prop-L1-cv}
Let $W\subset E$ be a Borel subset, and let $B_1 \subset B_2 \subset \cdots B_n\subset \cdots \subset W $ be an increasing exhausting sequence of bounded Borel subsets of  $W$.
  Fix any positive function $\psi: E\setminus W \rightarrow (0, 1]$ such that $\psi^{1/2}K \psi^{1/2} $ is of trace class and for any bounded subset $B\subset E$, we have $\inf_{x\in B} \psi(x)>0$.
  Then
  \begin{align}\label{psi-kernel}
  \big( \psi^{1/2}  K^{[X, \, B_n]} \psi^{1/2}   \big)_{n\in \N}
  \end{align}
  is an $\mathscr{L}_1( L^2(E\setminus W, \mu))$-valued martingale which is bounded in $L^2(\Conf(E), \PP; \mathscr{L}_1(L^2(E\setminus W, \mu)))$.
In particular, the sequence converges in $L^1(\Conf(E), \PP; \mathscr{L}_1(L^2( E\setminus W, \mu)))$.
\end{proposition}
\begin{proof}
It suffices to show that the sequence  \eqref{psi-kernel}  is bounded in $L^2(\Conf(E), \PP; \mathscr{L}_1(L^2(E\setminus W, \mu)))$. Indeed, we have
\begin{align*}
\big\|\psi^{1/2}  K^{[X, \, B_n]} \psi^{1/2}\big\|_{\mathscr{L}_1(L^2(E\setminus W, \mu))} = \tr\big(\psi^{1/2}  K^{[X, \, B_n]} \psi^{1/2} \big)  = \E_{\PP_K}\big(\sum_{x\in X}\psi(x)  \big| X, B_n \big).
\end{align*}
By Proposition \ref{prop-L2}, we get the desired $L^2(\Conf(E), \PP; \mathscr{L}_1(L^2(E\setminus W, \mu)))$-boundedness of the sequence \eqref{psi-kernel}.
\end{proof}

\begin{remark*}
Let $\mathscr{B}(W)$ be the directed set of bounded measurable subsets of  $W$, ordered by set-inclusion. Then the  set-indexed family $\big( \chi_{E\setminus W}  K^{[X, \, B]} \chi_{E\setminus W}  \big)_{B \in \mathscr{B}(W)}$
is a set-indexed martingale adapted to the filtration $(\mathcal{F}(B))_{B \in \mathscr{B}(W)}$. By virtue of Proposition \ref{prop-L1-cv},  for any positive function $\psi: E\setminus W \rightarrow (0, 1]$ such that $\psi^{1/2}K \psi^{1/2} $ is of trace class and for any bounded subset $B\subset E$, we have $\inf_{x\in B} \psi(x)>0$,  the set-indexed martingale
\begin{align*}
\Big( \psi^{1/2} K^{[X, \, B]} \psi^{1/2}  \Big)_{B\in \mathscr{B}(W)}
\end{align*}
converges in $L^1(\Conf(E), \PP_K; \mathscr{L}_{1}(L^2(E\setminus W, \mu)))$.
\end{remark*}

\section{Proof of Theorem \ref{mainthm-DT}}

Recall  that  we have fixed a realization of our kernel, namely, a Borel function $K(x,y)$  defined on the set $E_0\times E_0$, where $\mu(E\setminus E_0)=0$.
In this section, we make the additional assumption that $K$ is an orthogonal projection onto a subspace $H\subset L^2(E, \mu)$.  Recalling (\ref{kxt}),
we fix a realization for each $h\in H$: namely, in such a way that  the equation
$h(x) = \langle h, K_x\rangle$ holds  for every $x\in E_0$ and every $h\in H$.
 Given any configuration $X \in \Conf(E)$ and a bounded Borel subset $B\subset E$, we set
$
 L(X)  : = \{ h \in H:      h\restriction_X \equiv  0 \} \an
\chi_B L (X) := \{  \chi_B h : h \in L(X)\} \subset L^2(E, \mu).
$
The subspace $L(X)$ is of course closed, but  $\chi_B L (X) $ need not be closed.

Fix an exhausting sequence $E_1 \subset \cdots \subset E_n \subset \cdots \subset E\setminus B$ of bounded Borel subsets of $E\setminus B$, and denote
$$
F_n = E \setminus (B \cup E_n).
$$
Since $B$ is bounded, we have
$
\mathscr{L}_{1, loc}(L^2(B, \mu)) = \mathscr{L}_{1}(L^2(B, \mu)).
$
By Lemma \ref{key-lem},  for $\PP_K$-almost every $X\in \Conf(E)$,  there exists a positive contraction $K^{[X, \, E\setminus B]} \in \mathscr{L}_{1}(L^2(B, \mu))$, such that
\begin{align}\label{berg-limit}
\chi_{B} K^{[X, \, E_n]} \chi_{B} \xrightarrow[\text{in $\mathscr{L}_{1}(L^2(B, \mu))$}]{n\to\infty}  K^{[X, \, E \setminus B]}
\end{align}
and
\begin{align}\label{berg-cond}
\PP_K(\cdot | X, E\setminus B) = \PP_{K^{[X, \, E \setminus B]}}.
\end{align}

\begin{lemma}\label{lem-id}
For $\PP_K$-almost every $X\in \Conf(E)$, we have $K^{[X, \, E \setminus B]} (\chi_B h) = \chi_B h$ for any $h\in L(X \cap  (E \setminus B))$.
\end{lemma}

\begin{proof}
For any $n\in \N$, since $E_n \subset E \setminus B$,  by definition, we have
$
L(X \cap (E \setminus B)) \subset  L(X \cap E_n).
$
Since $E_n$ is bounded and $E\setminus E_n = B \cup F_n$, the operator $K^{[X, E_n]}$ is the orthogonal projection from $L^2(E,\mu)$ onto the closure of the subspace
$
\chi_{E \setminus E_n} L(X \cap E_n)  =  \chi_{B \cup F_n} L (X \cap E_n).
$
By \eqref{berg-limit}, for $\PP_K$-almost every $X\in \Conf(E)$, the limit relation 
\begin{align*}
K^{[X, \, E \setminus B]} (\chi_B h) & =\lim_{n\to\infty} \big(\chi_{B} K^{[X, \, E_n]} \chi_{B}\big)(\chi_B h)  = \lim_{n\to\infty} \chi_{B} K^{[X, \, E_n]}(\chi_B h)
\end{align*}
holds for any $h \in L(X \cap (E \setminus B))$. 
Using the equalities  $\chi_B h = \chi_{B \cup F_n} h  - \chi_{F_n} h$,
$
K^{[X, E_n]} (\chi_{B\cup F_n} h ) = \chi_{B\cup F_n} h,
$
and the relation $$
\| \chi_{B} K^{[X, \, E_n]}(\chi_{F_n} h) \|_2 \le \| \chi_{F_n} h \|_2 \xrightarrow{n\to\infty} 0,
$$
we obtain the desired equality 
\[
K^{[X, \, E \setminus B]} (\chi_B h)  =\lim_{n\to\infty} \chi_{B} K^{[X, \, E_n]}(\chi_{B\cup F_n} h - \chi_{F_n} h) = \chi_B h - \lim_{n\to\infty}  \chi_{B} K^{[X, \, E_n]}(\chi_{F_n} h)= \chi_B h.
\]
\end{proof}

\begin{lemma}\label{lem-good-rel}
Let $\PP$ be a point process on $E$.  Then for any bounded Borel subset $B\subset E$, we have
\begin{align}\label{non-vanish-con}
\PP(\#_B = \#(X \cap B) | X, B^c) > 0 \quad \text{for $\PP$-almost every $X \in \Conf(E)$}.
\end{align}
\end{lemma}

\begin{proof}

First of all, decomposing $X=Y\cup Z$, $Y\in\Conf(B)$, $Z\in\Conf(B^c)$, we can rewrite the statement as follows:
\begin{equation}\label{eq:VarLem72}
\Prob\bigl(\{W\in \Conf(B): \#(W)=\#(Y)\}|Z,B^c\bigr)>0
\end{equation}
for $(\pi_{B^c})_*(\mathbb P)$-almost every $Z\in\Conf(B^c)$ and $\Prob({}\cdot{}|Z, B^c)$-almost every $Y\in\Conf(B)$.
We make a simple general claim: given an integer-valued measurable function~$f$ on a probability space $(\Omega, \mathbf P)$, for $\mathbf P$-almost every $y\in\Omega$ we have ${\mathbf P}\{x:f(x)=f(y)\}>0$.
Indeed, if $N=\{n\in\mathbb Z: \mathbf P\{x:f(x)=n)\}=0\}$, then the relation $\mathbf P\{x:f(x)=f(y)\}>0$ fails only if $f(y)\in N$, and
$$
\mathbf P\{y:f(y)\in N\}=\sum_{n\in N}\mathbf P\{y:f(y)=n\}=0.
$$
For $(\pi_{B^c})_*(\mathbb P)$-almost every $Z\in\Conf(B^c)$,  by taking $\Omega=\Conf(B)$, $\mathbf P=\Prob({}\cdot{}|Z,B^c)$, $f=\#_B$, we obtain \eqref{eq:VarLem72}.
\end{proof}

\begin{proof}[Proof of Theorem \ref{mainthm-DT}]
Fix a countable dense subset $T$ of $E$  and let $S_n$ be an enumeration of balls with rational radii centred at $T$:
\begin{align}\label{s-n-collection}
\{S_n: n \in \N\} = \{B(x, q): x\in T, q \in \Q\}.
\end{align}
Since the family \eqref{s-n-collection} is countable, by Lemmata \ref{lem-id} and  \ref{lem-good-rel}, there exists  a measurable subset $\mathcal{A} \subset \Conf(E)$ such that  
\begin{itemize}
\item  $\PP_K(\mathcal{A}) = 1$;
\item for all $X \in \mathcal{A}$ and all $n\in \N$,  the conditional measures $\PP_K(\cdot| X, S_n^c)$ and conditional kernels $K^{[X, \, S_n^c]}$ are defined and satisfy
\begin{align}\label{eq-cond}
\PP_K(\cdot| X, S_n^c)  = \PP_{K^{[X, \, S_n^c]}};
\end{align}
\item for all $X\in \mathcal{A}$ and all $n\in \N$, we have 
\begin{align}\label{s-n-eigenn-f}
K^{[X, \,   S_n^c ]} (\chi_{S_n} h) = \chi_{S_n} h,  \quad \text{for any $h\in L(X \cap  S_n^c)$};
\end{align}
\item for all $X\in \mathcal{A}$ and all $n\in\N$.
\begin{align}\label{pos-pb}
\PP_K(\#_{S_n} = \#(X \cap S_n) | X, S_n^c) >0. 
\end{align}
\end{itemize}

We now show that the above measurable subset $\mathcal{A} \subset \Conf(E)$ satisfies the desired property, that is, $L(X) = \{0\}$ for any $X \in \mathcal{A}$.
Take a fixed configuration $X \in \mathcal{A}$ and assume, by contradiction, that there exists $h_0\in L(X)$, $h_0 \ne 0$. Clearly, since $X$ is a discrete countable subset,  there exists  $n_0\in \N$ such that 
\[
h_0\restriction_{S_{n_0}} \ne  0 \an  X \cap S_{n_0} = \emptyset.
\]
The relation $h_0\restriction_{S_{n_0}} \ne  0 $  implies that $\chi_{S_{n_0}} h_0 \ne 0$ and the relation $X \cap S_{n_0}  = \emptyset$ implies that $L(X)  =  L(X \cap S_{n_0}^c)$ and hence $h_0 \in L(X \cap S_{n_0}^c)$. Using the assumption \eqref{s-n-eigenn-f} on $\mathcal{A}$,   the non-zero function  $\chi_{S_{n_0}} h_0$ satisfies
$
K^{[X, \, S_{n_0}^c]}(\chi_{S_{n_0}} h_0) = \chi_{S_{n_0}} h_0,
$
 whence  $1$ is an eigenvalue of the operator $K^{[X, \, S_{n_0}^c]}$. In particular, this implies
\[
\det(1 - K^{[X, \, S_{n_0}^c]}) =0.
\]
On the other hand, the relations \eqref{eq-cond}, \eqref{pos-pb} together with the gap probability formula \eqref{gap-pb} imply that
\[
 \det (1 - K^{[X, \, S_{n_0}^c]})  = \PP_{K^{[X, \, S_{n_0}^c]}} (\#_{S_{n_0}} =0) = \PP_K(\#_{S_{n_0}} =0 | X, S_{n_0}^c)  = \PP_K(\#_{S_{n_0}} = \#(X \cap S_{n_0}) | X, S_{n_0}^c)>0.
\]
 We thus obtain a contradiction and Theorem \ref{mainthm-DT} is proved completely.
\end{proof}

\section{Triviality of the  tail $\sigma$-algebra: proof of Theorem \ref{main-thm3}}

\begin{definition}\label{def-unbdd}
Fix any increasing exhausting sequence $D_1 \subset \cdots \subset D_n \subset \cdots \subset E$ of bounded Borel subsets fo $E$.  For any Borel subset $W\subset E$, set
\[
K^{[X, \, W]}: = \lim_{n\to\infty} \chi_{E\setminus W} K^{[X,\, W\cap D_n]} \chi_{E\setminus W}.
\]
\end{definition}
The convergence takes place in $\mathscr{L}_{1, loc}(L^2(E,\mu))$ by Proposition \ref{prop-RNP}.The kernel $K^{[X, \, W]}$ is well-defined  for  $\PP_K$-almost every $X$. For fixed $W$, the limit almost surely is independent of the choice of the sequence $(D_n)_{n=1}^\infty$.

\begin{proposition}\label{prop-re-mart}
Fix a bounded Borel subset $B \subset E$ and let $E\setminus B \supset W_1 \supset \cdots \supset W_n \supset \cdots$ be any decreasing sequence of Borel subsets.  Then $\big(\chi_B K^{[X, \, W_n]} \chi_B\big)_{n\in\N}$ is an $(\mathcal{F}(W_n))_{n\in\N}$-adapted reverse martingale defined on the probability space $(\Conf(E), \mathcal{F}(E), \PP_K)$.
\end{proposition}

\begin{proof}
It suffices to prove that for any $\phi\in L^2(B, \mu)$, the sequence $\big(\langle  K^{[X, \, W_n]} \phi, \phi\rangle  \big)_{n\in\N}$
is an $(\mathcal{F}(W_n))_{n\in\N}$-adapted reverse martingale defined on the probability space $(\Conf(E), \mathcal{F}(E), \PP_K)$. By definition, for any $n\in\N$, we have
\begin{align}\label{as-conv}
\langle  K^{[X, \, W_n]} \phi, \phi\rangle  = \lim_{k \to\infty} \langle  K^{[X, \, W_n\cap D_k]} \phi, \phi\rangle, \quad \text{$\PP_K$-almost surely.}
\end{align}
Since all  the operators $K^{[X, \, W_n]}$ are contractive, by the bounded convergence theorem, the convergence  \eqref{as-conv} takes place in $L^1(\PP_K)$ as well. Fix an natural number $n\in\N$.  For any $\varepsilon>0$, let $k\in\N$ be large enough in such a way  that
\begin{equation}
\left\| \langle  K^{[X, \, W_n]} \phi, \phi\rangle - \langle  K^{[X, \, W_n\cap D_k]} \phi, \phi\rangle\right\|_{L^1(\PP_K)}\le \varepsilon; \
\left\| \langle  K^{[X, \, W_{n+1}]} \phi, \phi\rangle - \langle  K^{[X, \, W_{n+1}\cap D_k]} \phi, \phi\rangle\right\|_{L^1( \PP_K)}\le \varepsilon.
\end{equation}
For fixed $n\in \N$, the sequence
\[
\Big(\E_{\PP_K} \Big[ \langle  K^{[X, \, W_n]} \phi, \phi\rangle  \Big| \mathcal{F}(W_{n+1} \cap D_k)\Big]\Big)_{k=1}^\infty
\]
is a  martingale that converges in $L^1$-norm to  $\E_{\PP_K} \Big[ \langle  K^{[X, \, W_n]} \phi, \phi\rangle  \Big| \mathcal{F}(W_{n+1})\Big]$. We can therefore choose $k$  large enough in such  a way that
\begin{align*}
\left\| \E_{\PP_K} \Big[ \langle  K^{[X, \, W_n]} \phi, \phi\rangle  \Big| \mathcal{F}(W_{n+1})\Big]-    \E_{\PP_K} \Big[ \langle  K^{[X, \, W_n]} \phi, \phi\rangle  \Big| \mathcal{F}(W_{n+1} \cap D_k)\Big] \right\|_{L^1(\PP_K)}\le \varepsilon.
\end{align*}
Since $W_{n+1} \cap D_k \subset W_{n} \cap D_k$ and $D_k$ is bounded,  Lemma \ref{lem-mart-prop} implies
\[
\E_{\PP_K} \Big[ \langle  K^{[X, \, W_n \cap D_k]} \phi, \phi\rangle  \Big| \mathcal{F}(W_{n+1} \cap D_k)\Big]  =  \langle  K^{[X, \, W_{n+1} \cap D_k]} \phi, \phi\rangle,
\]
whence
\begin{multline}
\left\| \E_{\PP_K} \Big[ \langle  K^{[X, \, W_n]} \phi, \phi\rangle  \Big| \mathcal{F}(W_{n+1})\Big]-  \langle  K^{[X, \, W_{n+1}]} \phi, \phi\rangle\right\|_{L^1(\PP_K)}\le \\
\le  2 \varepsilon +  \left\|  \E_{\PP_K} \Big[ \langle  K^{[X, \, W_n]} \phi, \phi\rangle  \Big| \mathcal{F}(W_{n+1} \cap D_k)\Big]  -  \langle  K^{[X, \, W_{n+1} \cap D_k]} \phi, \phi\rangle\right\|_{L^1(\PP_K)}
\le \\ \le 3 \varepsilon +  \left\|  \E_{\PP_K} \Big[ \langle  K^{[X, \, W_n \cap D_k]} \phi, \phi\rangle  \Big| \mathcal{F}(W_{n+1} \cap D_k)\Big]  -  \langle  K^{[X, \, W_{n+1} \cap D_k]} \phi, \phi\rangle\right\|_{L^1( \PP_K)} =3 \varepsilon,
\end{multline}
and we obtain the desired reverse martingale relation
$
\E_{\PP_K} \Big[ \langle  K^{[X, \, W_n]} \phi, \phi\rangle  \Big| \mathcal{F}(W_{n+1})\Big] =  \langle  K^{[X, \, W_{n+1}]} \phi, \phi\rangle.
$
\end{proof}

\begin{lemma}\label{lem-var}
For any bounded Borel subset $B\subset E$ and $\phi \in L^2(B^c, \mu)$, we have
\begin{align}\label{cont-var}
\Var_{\PP_K} \big[\langle K^{[X, \, B]} \phi, \phi \rangle \big] \le   \| \phi\|_2^2\cdot  \|\chi_{B}  K \phi \|_2^2,
\end{align}
where $\| \cdot\|_2$ is the Hilbert norm on $L^2(E, \mu)$.
\end{lemma}

We first prove Lemma \ref{lem-var}  when $K$ is an orthogonal projection. This part of the proof is similar to the argument of Benjamini, Lyons, Peres and Schramm
\cite[Lemma 8.6]{BLPS} and Lyons \cite[Lemma 7.18]{DPP-L}.  The  proof of  Lemma \ref{lem-var} in full generality proceeds by reduction to the case of projections (the usual argument of extending the phase space must be slightly modified
  in the continuous setting) and is postponed to the end of the section.

\begin{proof}[Proof of Lemma \ref{lem-var} when $K$ is an orthogonal projection]
By homogeneity, we may assume that
$
\| \phi\|_2 \le 1.
$
Since $K$ is an orthogonal projection, by \cite[Proposition 2.4]{BQS}, so is $K^{[X, B]}$ for $\PP_K$-almost every $X\in \Conf(E)$. By Proposition \ref{prop-1-step}, we have
\begin{multline}
\Var_{\PP_K} \big[\langle K^{[X, \, B]} \phi, \phi \rangle \big] = \E_{\PP_K}\Big |  \big\langle (K^{[X, \, B]}  - \chi_{B^c} K \chi_{B^c}) \phi, \phi  \big\rangle \Big|^2  \le \E_{\PP_K} \Big(\big \|  (K^{[X, \, B]}  - \chi_{B^c} K \chi_{B^c}) \phi \big\|_2^2   \Big)=
\\
= \E_{\PP_K}\Big  ( \|   K^{[X, \, B]}  \phi\|_2^2   -  \langle K^{[X, \, B]}  \phi, \chi_{B^c} K \chi_{B^c} \phi \rangle - \langle \chi_{B^c} K \chi_{B^c} \phi, K^{[X, \, B]}  \phi\rangle + \|  \chi_{B^c} K \chi_{B^c} \phi \|_2^2  \Big)=
\\
 =   \E_{\PP_K}\Big  (  \langle K^{[X, \, B]}  \phi, \phi \rangle  -  \langle K^{[X, \, B]}  \phi, \chi_{B^c} K \chi_{B^c} \phi \rangle - \langle \chi_{B^c} K \chi_{B^c} \phi, K^{[X, \, B]}  \phi\rangle + \|  \chi_{B^c} K \chi_{B^c} \phi \|_2^2  \Big)=
\\
 = \langle  \chi_{B^c} K \chi_{B^c} \phi, \phi\rangle -  \|  \chi_{B^c} K \chi_{B^c} \phi \|_2^2 =  \langle K\phi, \phi\rangle  - \| \chi_{B^c} K\phi\|_2^2
 = \| K\phi\|_2^2 - \| \chi_{B^c} K\phi\|_2^2  = \| \chi_B K \phi\|_2^2.
\end{multline}
\end{proof}

\begin{proposition}\label{prop-tail-re-mart}
 Fix any $\ell \in\N$. Then $\big(\chi_{D_\ell}K^{[X, \,  E \setminus D_{n+\ell}]} \chi_{D_\ell}\big)_{n\in\N}$ is an $(\mathcal{F}(E\setminus D_{n+\ell}))_{n\in\N}$-adapted reverse martingale defined on the probability space $(\Conf(E), \mathcal{F}(E), \PP_K)$, and we have
\begin{align}\label{to-K}
 \chi_{D_\ell}K^{[X, \,  E \setminus D_{n+\ell}]} \chi_{D_\ell} \xrightarrow[\text{$\PP_K$-a.s. in $\mathscr{L}_1(L^2(E, \mu)$ and in $L^2(\PP_K; \mathscr{L}_1(L^2(E, \mu)))$}]{n\to\infty} \chi_{D_\ell} K \chi_{D_\ell}.
\end{align}
For any $\ell \in \N$, we have
\begin{align}\label{tail-marginal}
 \E_{\PP_K}\Big[ \PP_K(\cdot | X, E\setminus D_\ell) \Big|   \bigcap_{n=1}^\infty \mathcal{F}(E\setminus D_{n+\ell})\Big]  = (\pi_{D_\ell})_{*}(\PP_K), \quad \text{$\PP_K$-almost surely}
\end{align}
and, for any $\mathcal{A} \in \mathcal{F}(D_\ell)$, we have
\begin{align}\label{mart-event}
\lim_{n\to\infty} \E_{\PP_K} \Big| \E_{\PP_K}\big[\chi_{\mathcal{A}}   \big|\mathcal{F}(E\setminus D_{n+\ell}) \big]  - \PP_{K} (\mathcal{A})  \Big| = 0.
\end{align}
\end{proposition}

\begin{proof}
The reverse martingale property of the sequence follows from Proposition \ref{prop-re-mart}.
Set
\begin{align}\label{tail-sigma}
\mathscr{T} : =  \bigcap_{n=1}^\infty \mathcal{F}(E\setminus D_{n+\ell}).
\end{align}
Since a Banach space valued {\it reverse} martingale converges (see, e.g., Pisier \cite[p. 34]{pisier-B-martingale}), we obtain
\[
\chi_{D_\ell}K^{[X, \,  E \setminus D_{n+\ell}]} \chi_{D_\ell} \xrightarrow[\text{$\PP_K$-a.s. in $\mathscr{L}_1(L^2(E, \mu)$ and in $L^2(\PP_K; \mathscr{L}_1(L^2(E, \mu)))$}]{n\to\infty}   \E_{\PP_K}\big[\chi_{D_\ell}K^{[X, \,  E \setminus D_{1+\ell}]} \chi_{D_\ell}  \big|   \mathscr{T}\big].
\]
Set
$$
G_\infty(X)=
\E_{\PP_K}\big[\chi_{D_\ell}K^{[X, \,  E \setminus D_{1+\ell}]} \chi_{D_\ell}  \big|   \mathscr{T}\big].
$$
In particular, for any  $\phi \in L^2(D_\ell, \mu)$ with $\|\phi \|_2 \le 1$, we have
 \[
 \langle G_\infty(X)\phi, \phi \rangle  =  \E_{\PP_K} \Big[ \langle  K^{[X, \, E\setminus D_{1+\ell} ]} \phi, \phi\rangle  \Big| \mathscr{T} \Big], \quad \text{$\PP_K$-almost surely}.
  \]
 By Definition \ref{def-unbdd} and the inequality $\big| \langle  K^{[X, \, (E\setminus D_{n+\ell} ) \cap D_k]} \phi, \phi\rangle\big| \le 1$, which holds $\PP_K$-almost surely,
 for any $n\in\N$, we have
 \begin{align}\label{finite-inf}
\langle  K^{[X, \, (E\setminus D_{n+\ell} ) \cap D_k]} \phi, \phi\rangle  \xrightarrow{k\to\infty} \langle  K^{[X, \, E\setminus D_{n+\ell} ]} \phi, \phi\rangle, \quad \text{$\PP_K$-almost surely and in $L^2(\PP_K)$.}
 \end{align}
 Similarly,
  \begin{align}\label{inf-final}
 \langle  K^{[X, \, E\setminus D_{n+\ell} ]} \phi, \phi\rangle \xrightarrow{n\to\infty}  \langle G_\infty(X)\phi, \phi \rangle, \quad \text{$\PP_K$-almost surely and in $L^2(\PP_K)$.}
 \end{align}
In particular, since $(E\setminus D_{1+\ell} ) \cap D_k$ are bounded for all $k\in\N$, we can apply Proposition  \ref{prop-1-step} to obtain
\[
\E_{\PP_K}  \langle G_\infty(X)\phi, \phi \rangle  =   \E_{\PP_K} \big[ \langle  K^{[X, \, E\setminus D_{1+\ell} ]} \phi, \phi\rangle  \big] =\lim_{k \to\infty}  \E_{\PP_K} \big[ \langle  K^{[X, \, (E\setminus D_{1+\ell} ) \cap D_k]} \phi, \phi\rangle \big]  = \langle K \phi, \phi\rangle.
\]
Now by Lemma \ref{lem-var}, we have
\[
\Var_{\PP_K}\Big(\langle  K^{[X, \, (E\setminus D_{n+\ell} ) \cap D_k]} \phi, \phi\rangle \Big) \le  \|\chi_{(E\setminus D_{n+\ell} ) \cap D_k} K \phi\|_2^2 \le \|\chi_{E\setminus D_{n+\ell} } K \phi\|_2^2.
\]
 The convergence \eqref{finite-inf}, \eqref{inf-final} yields
 \begin{align*}
& \Var_{\PP_K}\Big(\langle  K^{[X, \, E\setminus D_{n+\ell} ]} \phi, \phi\rangle \Big) = \lim_{k \to\infty} \Var_{\PP_K}\Big(\langle  K^{[X, \, (E\setminus D_{n+\ell} ) \cap D_k]} \phi, \phi\rangle \Big)  \le  \|\chi_{E\setminus D_{n+\ell} } K \phi\|_2^2;
 \\
 & \Var_{\PP_K}\Big(  \langle G_\infty(X)\phi, \phi \rangle   \Big)  = \lim_{n\to\infty} \Var_{\PP_K}\Big(\langle  K^{[X, \, E\setminus D_{n+\ell} ]} \phi, \phi\rangle \Big)  \le \limsup_{n\to\infty}   \|\chi_{E\setminus D_{n+\ell} } K \phi\|_2^2=0.
 \end{align*}
 Consequently,  we have $\langle G_\infty(X)\phi, \phi \rangle = \langle  K \phi, \phi\rangle, \PP_K$-almost surely.
Since $ \chi_{D_\ell} G_\infty(X) \chi_{D_\ell} = G_\infty(X)$ and since $\phi$ is arbitrarily chosen from the separable unit sphere in $L^2(D_\ell, \mu)$, we obtain the desired equality
\[
G_\infty(X) = \chi_{D_\ell} K \chi_{D_\ell}, \quad \text{$\PP_K$-almost surely}.
\]

Finally, Proposition \ref{prop-mart-measure} implies that
\[
(\pi_{D_\ell})_{*}[ \PP_K(\cdot| X, E\setminus D_{n+\ell} )]= \E_{\PP_K}\Big[ \PP_K(\cdot | X, E\setminus D_\ell) \Big| \mathcal{F}(E\setminus D_{n+\ell})\Big],  \quad \text{$\PP_K$-almost surely,}
\]
and
\begin{align}\label{tv-conv}
(\pi_{D_\ell})_{*}[ \PP_K(\cdot| X, E\setminus D_{n+\ell} )] \xrightarrow[\text{weakly}]{n\to\infty}  \E_{\PP_K}\Big[ \PP_K(\cdot | X, E\setminus D_\ell) \Big| \mathscr{T}\Big], \quad \text{$\PP_K$-almost surely}.
\end{align}
But the convergence \eqref{to-K} implies that
\begin{align}\label{w-conv}
(\pi_{D_\ell})_{*}[ \PP_K(\cdot| X, E\setminus D_{n+\ell} )]=   \PP_{\chi_{D_\ell} K^{[X, \, E\setminus D_{n+\ell}]} \chi_{D_\ell}  } \xrightarrow[weakly]{n\to\infty}  \PP_{\chi_{D_\ell}  K \chi_{D_\ell} } = (\pi_{D_\ell})_{*}(\PP_K), \quad \text{$\PP_K$-almost surely}.
\end{align}
Now \eqref{tv-conv} and \eqref{w-conv}  yield  \eqref{tail-marginal}.
Martingale convergence  for a bounded random variable  implies  \eqref{mart-event}.
\end{proof}

\begin{proof}[Proof of Theorem \ref{main-thm3}]
Take $D_n: = B_n$.  We prove that the $\sigma$-algebra $\mathscr{T}$ in \eqref{tail-sigma} is trivial with respect to $\PP_K$.  Take an event $\mathcal{A} \in \mathscr{T}$.  For $\varepsilon >0$, find  $\ell \in \N$ large enough and $\mathcal{A}_1 \in \mathcal{F}(D_\ell)$ such that $\PP_K (\mathcal{A}_1 \Delta \mathcal{A}) < \varepsilon/3$.
By \eqref{mart-event}, we have
\[
\lim_{n\to\infty} \E_{\PP_K} \Big| \E_{\PP_K}\big[\chi_{\mathcal{A}_1}   \big|\mathcal{F}(E\setminus D_{n+\ell}) \big]  - \PP_{K} (\mathcal{A}_1)  \Big| = 0.
\]
 Now find $n \in \N$ large enough in such a way that
\begin{align*}
\E_{\PP_K} \Big| \E_{\PP_K}\big[\chi_{\mathcal{A}_1}   \big|\mathcal{F}(E\setminus D_{n+\ell}) \big]  - \PP_{K} (\mathcal{A}_1)  \Big|  \le \varepsilon/3.
\end{align*}
It follows that  for any $\mathcal{A}_2 \in \mathcal{F} (E\setminus D_{n+\ell})$, we have
\begin{multline}
 | \PP_K(\mathcal{A}_1 \cap \mathcal{A}_2)  - \PP_K(\mathcal{A}_1) \PP_K(\mathcal{A}_2) | = \Big| \E_{\PP_K}    \Big(  \chi_{\mathcal{A}_2} \E_{\PP_K}\big[\chi_{\mathcal{A}_1}   \big|\mathcal{F}(E\setminus D_{n+\ell}) \big]    \Big)   -   \E_{\PP_K}    \Big(  \chi_{\mathcal{A}_2}  \PP_{K} (\mathcal{A}_1)     \Big)   \Big|=
 \\
 = \Big| \E_{\PP_K}    \Big(  \chi_{\mathcal{A}_2} \Big[\E_{\PP_K}\big[\chi_{\mathcal{A}_1}   \big|\mathcal{F}(E\setminus D_{n+\ell}) \big]   - \PP_{K} (\mathcal{A}_1) \Big]   \Big)     \Big|
  \le \E_{\PP_K} \Big(\Big| \E_{\PP_K}\big[\chi_{\mathcal{A}_1}   \big|\mathcal{F}(E\setminus D_{n+\ell}) \big]  - \PP_{K} (\mathcal{A}_1)  \Big|\Big)  \le \varepsilon/3.
\end{multline}
Finally, we obtain
\begin{align*}
 | \PP_K(\mathcal{A} \cap \mathcal{A}_2)  - \PP_K(\mathcal{A}) \PP_K(\mathcal{A}_2) | \le 2  \PP_K (\mathcal{A}_1 \Delta \mathcal{A})  +   | \PP_K(\mathcal{A}_1 \cap \mathcal{A}_2)  - \PP_K(\mathcal{A}_1) \PP_K(\mathcal{A}_2) | \le \varepsilon.
\end{align*}
Taking $\mathcal{A}_2 = \mathcal{A}$, we obtain $\PP_K(\mathcal{A}) = (\PP_K(\mathcal{A}))^2$, whence $\PP_K(\mathcal{A})$ is either $0$ or $1$, as desired.
\end{proof}

\begin{proof}[Proof of Lemma \ref{lem-var} in the general case]

Fix a bounded Borel subset $B\subset E$ and a function $\phi \in L^2(E\setminus B, \mu)$ such that
$
\| \phi\|_2 = 1.
$
 Recalling \eqref{def-R}, set
\[
R(K, B, \phi) =  (\phi \otimes \overline{\phi} + \chi_B) K (\phi \otimes \overline{\phi} + \chi_B).
\]
By Lemma \ref{main-local},
\[
\langle R(K, B, \phi) ^{[X, \, B]} \phi, \phi \rangle =\langle K^{[X, \, B]} \phi, \phi \rangle, \quad \text{for $\PP_K$-almost every $X\in \Conf(E)$.}
\]
By definition, we have $K^{[X, \, B]} =K^{[X\cap B, \, B]}$ and similarly $R(K, B, \phi) ^{[X, \, B]}= R(K, B, \phi) ^{[X\cap B, \, B]}$. In particular, we have
\[
\langle R(K, B, \phi) ^{[X, \, B]} \phi, \phi \rangle =\langle K^{[X, \, B]} \phi, \phi \rangle \quad \text{for $(\pi_B)_{*}(\PP_K) = \PP_{\chi_B K \chi_B}$-almost every $X\in \Conf(B)$};
\]
\begin{align}\label{many-var}
\Var_{\PP_{K}} \big[\langle K^{[X, \, B]} \phi, \phi \rangle \big] = \Var_{\PP_{\chi_B K \chi_B}} \big[\langle K^{[X, \, B]} \phi, \phi \rangle \big] =  \Var_{\PP_{\chi_B R(K, B, \phi)  \chi_B}}  \langle R(K, B, \phi) ^{[X, \, B]} \phi, \phi \rangle.
\end{align}

\begin{proposition}[{See \cite[Section 3.3]{Lyons-ICM}}]\label{claim-dilation}
Let ${\bf m}$ be the counting measure on $\N$. There  exists a locally trace class orthogonal projection operator $\widetilde{K} \in \mathscr{L}_{1, loc}(L^2(E\sqcup \mathbb{N}, \mu \oplus {\bf m}))$ such that $K = \chi_E \widetilde{K} \chi_E$.
\end{proposition}

\begin{proof}
The canonical orthogonal projection dilation of $K$ on $L^2(E, \mu) \oplus L^2(E, \mu)$ is given by the formula
\[
\left[
\begin{array}{cc}
K & \sqrt{K - K^2}
\\
\sqrt{K-K^2} & 1-K
\end{array}
\right],
\]
but it is not in general locally trace class.
Since $L^2(E, \mu)$ is  separable  and all infinite dimensional separable Hilbert spaces are isometrically isomorphic,  there exists a unitary operator $U: L^2(E, \mu) \rightarrow  \ell^2(\N) = L^2(\N, {\bf m})$, and we set
\begin{align*}
\widetilde{K}: =
\left[
\begin{array}{cc}
1  & 0
\\
0 & U^{-1}
\end{array}
\right]
\left[
\begin{array}{cc}
K & \sqrt{K - K^2}
\\
\sqrt{K-K^2} & 1-K
\end{array}
\right]
\left[
\begin{array}{cc}
1  & 0
\\
0 & U
\end{array}
\right] .
\end{align*}
\end{proof}

Since $\widetilde{K}$ is an orthogonal projection,  for any bounded Borel subset $B\subset  E$, which is of course also a subset of $E\sqcup \N$,  and any $\phi \in L^2(  E\setminus B, \mu)$, which  of course also lies in $L^2((E\sqcup \N) \setminus B, \mu \oplus {\bf m})$, we have
\begin{align*}
\Var_{\PP_{\widetilde{K}}} \big[\langle \widetilde{K}^{[X, \, B]} \phi, \phi \rangle \big] \le  \|\chi_{B}  \widetilde{K} \phi \|_2^2.
\end{align*}
For the term on the right hand side, we have
\begin{align}\label{eq-f}
\chi_{B}  \widetilde{K} \phi= \chi_B K \phi.
\end{align}
Since  $\phi \otimes \overline{\phi} + \chi_B = (\phi \otimes \overline{\phi} + \chi_B)\chi_E$, we have
\[
R(\widetilde{K}, B, \phi) =  (\phi \otimes \overline{\phi} + \chi_B) \widetilde{K} (\phi \otimes \overline{\phi} + \chi_B) =  (\phi \otimes \overline{\phi} + \chi_B) K (\phi \otimes \overline{\phi} + \chi_B) = R(K, B, \phi).
\]
It follows that
\[
\langle \widetilde{K}^{[X, \, B]} \phi, \phi \rangle   \xlongequal{\text{$\PP_{\chi_B \widetilde{K} \chi_B}$-a.s.}} \langle R(\widetilde{K}, B, \phi) ^{[X, \, B]} \phi, \phi \rangle  =  \langle R(K, B, \phi) ^{[X, \, B]} \phi, \phi \rangle     \xlongequal{\text{$\PP_{\chi_B K \chi_B}$-a.s.}} \langle K^{[X, \, B]} \phi, \phi \rangle.
\]
The equality $\chi_B \widetilde{K}\chi_B =\chi_B K \chi_B$ implies  the equality $\PP_{\chi_B \widetilde{K} \chi_B} = \PP_{\chi_B {K} \chi_B}$ , and we have
\begin{align}\label{eq-var}
\Var_{\PP_{K}} \big[\langle K^{[X, \, B]} \phi, \phi \rangle \big]  = \Var_{\PP_{\chi_B K\chi_B}} \big[\langle K^{[X, \, B]} \phi, \phi \rangle \big]  = \Var_{\PP_{ \chi_B \widetilde{K}
\chi_B}} \big[\langle \widetilde{K}^{[X, \, B]} \phi, \phi \rangle \big] = \Var_{\PP_{\widetilde{K}}} \big[\langle \widetilde{K}^{[X, \, B]} \phi, \phi \rangle \big].
\end{align}
Combining \eqref{eq-f} and \eqref{eq-var}, we obtain the desired inequality  \eqref{cont-var}.
\end{proof}

\section{Martingales corresponding to conditional processes}

\begin{proposition}\label{prop-bdd-cond}
Let $B\subset E$ be a bounded Borel subset. If $\PP$ is a simple point process on $E$ admitting correlation  measures of all orders, then
$  \PP(\cdot | X, B) = \overline{\PP^{X\cap B}\restriction_{\Conf(B^c)}}$ for $\PP$-almost every $X \in \Conf(E)$.
\end{proposition}

\begin{proof}
Let $\Conf_n(E) =\{X \in \Conf(E): \# X = n\}$ and similarly define $\Conf_n(B)$. By the natural map $E^n \rightarrow \Conf_n(E)$ defined by
$(x_1, \cdots, x_n) \mapsto  \{x_1, \cdots, x_n\}$,
we define a measure $\rho^\#_{n, \PP}$ on $\Conf_n(E)$  as the push-forward measure of  the correlation measure $\rho_{n, \PP}$ and define a $\sigma$-finite measure $\mathscr{C}_{n, \PP}^\#$  on $\Conf_n(E)\times \Conf(E)$ as the push-forward measure of $n$-th order Campbell measure $\mathscr{C}_{n, \PP}^{!}$. The formula \eqref{def-Palm}  implies that
\begin{align}\label{dis-campbell}
\mathscr{C}_{n, \PP}^\# (\dd  \mathfrak{p} \times \dd X_1) = \rho_{n, \PP}^\#(\dd \mathfrak{p})  \PP^{\mathfrak{p}} ( \dd  X_1).
\end{align}
By convention,  we set $ \rho_{0, \PP}^\#(\dd \mathfrak{p}) :=  \delta_{\emptyset}$ and
   $C_{0, \PP}^\#:= \delta_\emptyset \otimes\PP$, where $\delta_\emptyset$ is the Dirac measure at the empty configuration $\emptyset$, i.e.,  the unique element $\emptyset \in \Conf_0(E)$.
  Equivalently, for any positive Borel function $H: \Conf_n(E) \times \Conf(E)\rightarrow\R^{+}$:
\begin{align*}
\int_{\Conf_n(E) \times \Conf(E)}   H(\mathfrak{p}, X_1)  \mathscr{C}_{n, \PP}^\# (\dd X_0 \times \dd X_1) =  \int_{\Conf(E)}   \Big[  {\sum_{x \in X^n}}\!\!\!^{\#}   H(\{x_1, \cdots, x_n\},  X\setminus \{x_1, \cdots, x_n\}) \Big]  \PP(\dd  X),
\end{align*}
where the summation ${\sum\!^{\#}}$ is taken over all {\it ordered $n$-tuples} $(x_1, \cdots, x_n)$ with distinct coordinates $x_1, \cdots, x_n \in X$. In particular, when $n=0$, this equality reads as: for any $H: \Conf_0(E) \times \Conf(E)\rightarrow\R^{+}$, we have
\begin{align*}
\int_{\Conf_0(E) \times \Conf(E)}   H(\mathfrak{p}, X_1)  \mathscr{C}_{0, \PP}^\# (\dd \mathfrak{p} \times \dd X_1) =  \int_{\Conf(E)}    H(\emptyset, X) \PP(\dd X).
\end{align*}

The boundedness of  $B\subset E$ implies that $\Conf(B) = \bigsqcup_{n=0}^\infty \Conf_n(B)$.  Hence
\[
\Conf(E)  \simeq \Conf(B) \times \Conf(B^c) =  \Big(  \bigsqcup_{n=0}^\infty \Conf_n(B) \Big)\times \Conf(B^c) =   \bigsqcup_{n=0}^\infty    \Big(   \Conf_n(B) \times \Conf(B^c)  \Big).
\]
For any $n =0, 1, 2, \cdots$, let $H: \Conf_n(E) \times \Conf(E)\rightarrow \R^{+}$ be any non-negative Borel function supported on the subset $\Conf_n(B) \times \Conf(B^c) \subset \Conf_n(E) \times \Conf(E)$. Then  for any configuration $X \in \Conf(E)$, we have
\[
{\sum_{x \in X^n}}\!\!\!^{\#}   H(\{x_1, \cdots, x_n\},  X\setminus \{x_1, \cdots, x_n\})  =   n!  \cdot  \chi_{\{\#(X \cap B) = n \}}  \cdot H( X \cap B,  X \cap B^c).
\]
When $n=0$, this equality reads as $H(\emptyset, X)  =  \chi_{\{X \cap B = \emptyset \}}  \cdot H(X\cap B,  X \cap B^c)$.
By definition of $\mathscr{C}_{n, \PP}^\#$, we get
\begin{align*}
&\int_{\Conf_n(E) \times \Conf(E)}   H(\mathfrak{p}, X_1)  \mathscr{C}_{n, \PP}^\# (\dd \mathfrak{p} \times \dd X_1)  = \int_{\Conf(E)}   \Big[{\sum_{x \in X^n}}\!\!\!^{\#}   H(\{x_1, \cdots, x_n\},  X\setminus \{x_1, \cdots, x_n\}) \Big]  \PP(\dd  X)
\\
& = n!\cdot  \int_{\Conf(E)}      \chi_{\{\# (X\cap B) = n \}}  \cdot H(X \cap B, X\cap B^c)  \PP(\dd X) =  n!\cdot  \int_{\Conf_n(B) \times \Conf(B^c)} H(\mathfrak{p}, X_1)   \PP_{B, B^c}(\dd \mathfrak{p} \times \dd X_1).
\end{align*}
The above equality, combined  with \eqref{dis-campbell}, yields
\begin{align*}
& \PP_{B, B^c}\restriction_{\Conf_n(B) \times \Conf(B^c)}  (\dd \mathfrak{p} \times \dd X_1)  = \frac{1}{n!}\mathscr{C}_{n, \PP}^\#\restriction_{\Conf_n(B) \times \Conf(B^c)}  (\dd \mathfrak{p} \times \dd X_1)
\\
& = \frac{1}{n!} \rho_{n, \PP}^\# \restriction_{\Conf(B)}(\dd \mathfrak{p})  \PP^{\mathfrak{p}}\restriction_{\Conf(B^c)} ( \dd  X_1) =  \frac{\PP^{\mathfrak{p}}(\Conf(B^c))  }{n!} \rho_{n, \PP}^\# \restriction_{\Conf(B)}(\dd \mathfrak{p})  \overline{\PP^{\mathfrak{p}}\restriction_{\Conf(B^c)}} ( \dd  X_1).
\end{align*}
Consequently,
\[
\PP_{B, B^c} (\dd \mathfrak{p} \times \dd X_1)  =\Big(  \sum_{n=0}^\infty  \frac{\PP^{\mathfrak{p}}(\Conf(B^c))  }{n!} \rho_{n, \PP}^\# \restriction_{\Conf(B)}(\dd \mathfrak{p})  \Big)\overline{\PP^{\mathfrak{p}}\restriction_{\Conf(B^c)}} ( \dd  X_1).
\]
This implies both the formula for  $\pi_B(\PP)(\dd \mathfrak{p})$ and the formula for $\PP(\dd X_1| \mathfrak{p}, B) = \PP_{B, B^c}(\dd X_1 | \mathfrak{p}, B)$:
\begin{align}
\pi_B(\PP) (\dd \mathfrak{p})  &=  \sum_{n=0}^\infty  \frac{\PP^{\mathfrak{p}}(\Conf(B^c))  }{n!} \rho_{n, \PP}^\#\restriction_{\Conf(B)}(\dd \mathfrak{p}) ; \label{restriction-correlation}
\\
  \PP(\dd X_1 |\mathfrak{p}, B) &= \overline{\PP^{\mathfrak{p}}\restriction_{\Conf(B^c)}} ( \dd  X_1), \quad \text{for $\pi_B(\PP)$-almost every $\mathfrak{p} \in \Conf(B)$}.
\end{align}
Hence we get the desired relation $\PP(\cdot |X, B) = \overline{\PP^{X \cap B}\restriction_{\Conf(B^c)}}$, for $\PP$-almost every $X \in \Conf(E)$.
\end{proof}

\begin{remark*}
Kallenberg \cite[Section 12.3]{Kallenberg} defined the {\it compound Campbell} measure of $\PP$  on $\Conf_{fin}(E)\times \Conf(E)$ by
\[
\mathscr{C}_\PP^\# (\dd \mathfrak{p} \times \dd X_1): =\sum_{n=0}^\infty \frac{1}{n!}\mathscr{C}_{n, \PP}^\#(\dd \mathfrak{p} \times \dd X_1),
\]
where $\Conf_{fin}(E) = \sqcup_{n=0}^\infty \Conf_{n}(E)$.
\end{remark*}

Let $\PP$ be a point process on $E$ and let $W \subset E$ be a Borel subset of $E$.
Let $W_1 \subset \cdots \subset W_n \subset \cdots \subset W$ be an increasing sequence of Borel subsets of $W$ such that $W  = \bigcup_{n=1}^\infty W_n.$

\begin{proposition}\label{prop-mart-measure}

  The sequence $\big(  ( \pi_{W^c})_{*}[\PP(\cdot| X, W_n)]\big)_{n\in\N}$ is an $(\mathcal{F}(W_n))_{n\in\N}$-adapted martingale defined on the probability space $(\Conf(E), \mathcal{F}(E), \PP)$. Moreover, we have
\begin{align}\label{cond-meas-mart}
 (\pi_{W^c})_{*}[\PP(\cdot |X,  W_n)] = \E_{\PP}\Big[\PP(\cdot| X, W) \Big| \mathcal{F}(W_n) \Big], \quad \text{for $\PP$-almost every $X\in \Conf(E)$.}
\end{align}
In particular, by martingale convergence theorem,  for all Borel subsets $\mathcal{A} \subset \Conf(W^c)$ and any $1\le p < \infty$,   we have
\begin{align}\label{cond-mart}
\Big(( \pi_{W^c})_{*}[\PP(\cdot |X,  W_n)] \Big)  (\mathcal{A})  \xrightarrow[\text{$\PP$-a.s. and in $L^p(\Conf(E), \PP)$} ]{n\to\infty} \PP (\mathcal{A}| X, W).
\end{align}
Moreover,  for $\PP$-almost every $X\in \Conf(E)$, we have
\begin{align}\label{cond-mart-weak}
(\pi_{W^c})_{*}[\PP(\cdot |X,  W_n)]   \xrightarrow[\text{weakly}]{n\to\infty} \PP (\cdot | X, W).
\end{align}
\end{proposition}

\begin{remark*}
In general,  the statement \eqref{cond-mart}  cannot be strengthened to the claim that for $\PP$-almost every $X\in \Conf(E)$, we have
$
\Big((\pi_{W^c})_{*}[\PP(\cdot |X,  W_n)] \Big)  (\mathcal{A})  \xrightarrow{n\to\infty} \PP (\mathcal{A}| X, W),$
for {\bf all} Borel subsets $\mathcal{A} \subset \Conf(W^c).$

\end{remark*}

We prepare a simple lemma.
Let $\Omega_i$, $i=1, \dots, n, \dots$, and $\Omega^*$ be standard Borel spaces.
Fix $n\in\N$ and denote
\[
x: = (x_i)_{i=1}^\infty \an t =: (x_i)_{i\ge n+1},
\]
while $z$ will stand for the coordinate on $\Omega^*$.
Let $Q(\dd x\times \dd z)$ be a Borel probability measure on $(\prod_{i=1}^\infty \Omega_i ) \times \Omega^{*}$.  For any $n\in\N$, let $q_n(x_1, \cdots, x_n; \dd z)$ be the marginal on $\Omega^*$ of the conditional measure $Q(\dd t \times  dz| x_1, \cdots, x_n)$.
\begin{lemma}\label{lem-meas-mart} We have
\[
q_n(x_1, \cdots, x_n; \dd z) = \E[Q(dz| x_1, \cdots, x_n, t) | x_1, \cdots, x_n].
\]
\end{lemma}
\begin{proof}
 Denote by $Q_n$ the marginal measure of $Q$ on $\Omega_1 \times \cdots \times \Omega_n$. Let $Q_\infty$ be  the marginal measure of $Q$ on $\prod_{i=1}^\infty \Omega_i $.
 By definition of conditional measures, we have
\begin{align*}
Q(\dd x  \times \dd z)& =Q_\infty( \dd x) Q(dz| x_1, \cdots, x_n, t);
\\
Q(\dd x  \times \dd z)& =   Q_{n} (\dd x_1 \cdots \dd x_n)  Q(\dd t \times  dz| x_1, \cdots, x_n).
\end{align*}
And also
\begin{align*}
 \E[Q(dz| x_1, \cdots, x_n, t) | x_1, \cdots, x_n] =  \int_{ t \in \prod_{i=n+1}^\infty \Omega_i}   Q(dz| x_1, \cdots, x_n, t)  Q_\infty(\dd t  | x_1, \cdots,  x_n).
\end{align*}
Since
\[
Q_\infty( \dd x )  = Q_{n} (\dd x_1 \cdots \dd x_n)  Q_\infty(\dd t  | x_1, \cdots,  x_n),
\]
we get
\[
Q(\dd x \times \dd z) =  Q_{n} (\dd x_1 \cdots \dd x_n)  Q_\infty(\dd t  | x_1, \cdots,  x_n)  Q(dz| x_1, \cdots, x_n, t).
\]
Consequently,
\begin{align*}
Q(\dd t \times  dz| x_1, \cdots, x_n) =  Q_\infty(\dd t  | x_1, \cdots,  x_n)  Q(dz| x_1, \cdots, x_n, t).
\end{align*}
By definition, we have
\begin{align*}
q_n(x_1, \cdots, x_n; \dd z) &= \int_{t \in \prod_{i=n+1}^\infty \Omega_i} Q(\dd t \times  dz| x_1, \cdots, x_n)
\\
& = \int_{t\in \prod_{i=n+1}^\infty \Omega_i} Q_\infty(\dd t | x_1, \cdots,  x_n)  Q(dz| x_1, \cdots, x_n, t )
\\
&   =  \E[Q(dz| x_1, \cdots, x_n, t) | x_1, \cdots, x_n].
\end{align*}
\end{proof}

\begin{proof}[Proof of Proposition \ref{prop-mart-measure}]
Apply   Lemma \ref{lem-meas-mart} to $\Omega_i=  \Conf(W_i \setminus W_{i-1})$.
\end{proof}

Given  a bounded non-negative Borel function $g: E\rightarrow \R^{+}$,  let $S_g: \Conf(E)\rightarrow \R^{+} \cup\{+\infty\}$ denote the linear statistics defined, for $Z\in \Conf(E)$, by the formula $S_g(Z)=\sum_{x\in Z} g(x)$.
Denote by  $\E_\PP(S_g |X,  W)$  the conditional expectation of $S_g$ with respect to the sigma-algebra $\mathcal{F}(W)$.
\begin{proposition}\label{prop-L2}
If $g\restriction_W \equiv 0$ and $\E_\PP(S_g^2) <\infty$, then the sequence
\begin{align}\label{mart-nb}
\Big(\E_\PP(S_g |X,  W_n)\Big)_{n\in\N}
\end{align}
is an $(\mathcal{F}(W_n))_{n\in\N}$-adapted  $L^2(\Conf(E), \PP)$-bounded martingale defined on the probability space $(\Conf(E), \mathcal{F}(E), \PP)$.
\end{proposition}

\begin{proof}
Since $g\restriction_W \equiv 0$, by \eqref{cond-meas-mart}, we have
\[
\E_\PP(S_g |X,  W_n)   =  \E_{\PP}\Big[  \E_\PP(S_g | X, W)\Big| \mathcal{F}(W_n) \Big], \quad \text{for $\PP$-almost every $X\in \Conf(E)$.}
\]
By Jensen's inequality, we have
\[
[\E_\PP(S_g |X,  W_n) ]^2   \le \E_\PP(S_g^2  |X,  W_n), \quad \text{for $\PP$-almost every $X\in \Conf(E)$.}
\]
Therefore, for any $n\in\N$,
\[
\E_\PP [\E_\PP(S_g |X,  W_n) ]^2   \le \E_\PP(S_g^2 )< \infty.
\]
\end{proof}

\section{Mixing for M{\"o}bius transformations acting on  $(\Conf(\D), \PP_{K_\D})$ and proof of Lemma \ref{non-unif-sep}}\label{sample}

For any $n\in\N$ and  any $\varepsilon >0$, we have
\[
\PP\Big( \#(Z(\frak{f}_\D) \cap \{z \in \D: |z| \le \varepsilon\}) \ge n\Big) > 0.
\]
To conclude the proof of Lemma \ref{non-unif-sep}, it suffices to establish the ergodicity of the distribution of $Z(\frak{f}_\D)$ under the group $\Aut(\D)$ of M\"obius transformations, in other words, the group of isometries of the Lobachevsky plane. We prove mixing for  hyperbolic and parabolic one-dimensional subgroups of $\Aut(\D)$.
\begin{lemma}\label{thm-mixing}
If $\gamma \in \Aut(\D)$ is either hyperbolic or parabolic, then the dynamical system $(\Conf(\D), \PP_{K_\D}, \gamma)$ is strongly mixing.
\end{lemma}

\begin{proof}[Proof of Lemma \ref{thm-mixing}]
Fix an increasing sequence $r_k$ in the unit open interval $(0, 1)$ such that  $\lim_k r_k = 1$. Let $\mathcal{A}, \mathcal{B}$ be any two fixed measurable subsets in $\Conf(\D)$. For any $\varepsilon>0$, there exist $\mathcal{A}_\varepsilon, \mathcal{B}_\varepsilon \subset \Conf(\D)$ and a compact subset $C_\varepsilon \subset \D$, such that $\mathcal{A}_\varepsilon, \mathcal{B}_\varepsilon$ are both $\mathcal{F}(C_\varepsilon)$-measurable and 
\begin{align}\label{red-cpt}
\PP_{K_\D}(\mathcal{A} \Delta \mathcal{A}_\varepsilon) \le \varepsilon, \quad \PP_{K_\D}(\mathcal{B} \Delta \mathcal{B}_\varepsilon) \le \varepsilon.
\end{align}
Since $\PP_{K_\D}$ is $\gamma$-invariant, we have 
\begin{align}\label{unif-diff-set}
\sup_{n\in\N} \Big| \PP_{K_\D} (\mathcal{A} \cap \gamma^{-n}(\mathcal{B})) - \PP_{K_\D}(\mathcal{A}_\varepsilon \cap\gamma^{-n}(\mathcal{B}_\varepsilon))\Big|\le 2 \varepsilon.
\end{align}
By the assumption on $\gamma$, for any $k \in \N$, there exists $n_k\in \N$, such that 
\[
\gamma^{-n}(C_\varepsilon) \cap \D_{r_k} = \emptyset, \quad \text{for all $n\ge n_k$}.
\]
It follows that for any $n \ge n_k$, we have 
\[
\PP_{K_\D}(\mathcal{A}_\varepsilon \cap \gamma^{-n}(\mathcal{B}_\varepsilon)) =  \E_{\PP_{K_\D}} \Big(   \chi_{\gamma^{-n}(\mathcal{B}_\varepsilon)}  \E_{\PP_{K_\D}}   \big[  \chi_{\mathcal{A}_\varepsilon}| \mathcal{F}(\D \setminus \D_{r_k})\big]\Big).
\]
Therefore, for any $k$, we have
\begin{align*}
& \limsup_{n\to\infty}  \Big|\PP_{K_\D}(\mathcal{A}_\varepsilon \cap \gamma^{-n}(\mathcal{B}_\varepsilon)) - \PP_{K_\D}(\mathcal{A}_\varepsilon) \PP_{K_\D}(\mathcal{B}_\varepsilon)\Big| 
\\
 =& \limsup_{n\to\infty} \Big|  \E_{\PP_{K_\D}} \Big(   \chi_{\gamma^{-n}(\mathcal{B}_\varepsilon)}  \E_{\PP_{K_\D}}   \big[  \chi_{\mathcal{A}_\varepsilon}| \mathcal{F}(\D \setminus \D_{r_k})\big]\Big) - \E_{\PP_{K_\D}} \Big(   \chi_{\gamma^{-n}(\mathcal{B}_\varepsilon)}  \E_{\PP_{K_\D}}   \big[  \chi_{\mathcal{A}_\varepsilon}\big]\Big) \Big|
\\
\le & \limsup_{n\to\infty}  \E_{\PP_{K_\D}} \Big| \chi_{\gamma^{-n}(\mathcal{B}_\varepsilon)}  \E_{\PP_{K_\D}}   \big[  \chi_{\mathcal{A}_\varepsilon}| \mathcal{F}(\D \setminus \D_{r_k})\big] -  \chi_{\gamma^{-n}(\mathcal{B}_\varepsilon)}  \E_{\PP_{K_\D}}   \big[  \chi_{\mathcal{A}_\varepsilon}\big] \Big|
\\
\le &  \E_{\PP_{K_\D}} \Big| \E_{\PP_{K_\D}}   \big[  \chi_{\mathcal{A}_\varepsilon}| \mathcal{F}(\D \setminus \D_{r_k})\big] - \E_{\PP_{K_\D}}   \big[  \chi_{\mathcal{A}_\varepsilon}\big] \Big|.
\end{align*}
 Theorem \ref{main-thm3} now implies
\[
\lim_{k\to \infty} \E_{\PP_{K_\D}} \Big| \E_{\PP_{K_\D}}   \big[  \chi_{\mathcal{A}_\varepsilon}| \mathcal{F}(\D \setminus \D_{r_k})\big] - \E_{\PP_{K_\D}}   \big[  \chi_{\mathcal{A}_\varepsilon}\big] \Big|  = 0
\]
and hence 
\begin{align}\label{mix-cpt}
\limsup_{n\to\infty}  \Big|\PP_{K_\D}(\mathcal{A}_\varepsilon \cap \gamma^{-n}(\mathcal{B}_\varepsilon)) - \PP_{K_\D}(\mathcal{A}_\varepsilon) \PP_{K_\D}(\mathcal{B}_\varepsilon)\Big|  =0.
\end{align}
Combining \eqref{red-cpt}, \eqref{unif-diff-set} and \eqref{mix-cpt}, we obtain 
\[
\lim_{n\to\infty} \PP_{K_\D}(\mathcal{A} \cap \gamma^{-n}(\mathcal{B})) = \PP_{K_\D}(\mathcal{A}) \PP_{K_\D}(\mathcal{B})
\]
and thus complete the proof of the strong mixing property of the dynamical system $(\Conf(\D), \PP_{K_\D}, \gamma)$.
\end{proof}

\begin{proof}[Proof of Lemma \ref{non-unif-sep}]
We need to show that almost surely, 
\[
\sup_{\gamma \in \Aut(\D)}  \# \Big(Z(\frak{f}_\D) \cap \gamma^{-1}(\D_\varepsilon) \Big) = \infty, \quad \text{for all $\varepsilon \in (0, 1) \cap \Q$}.
\]
Since $(0, 1)\cap \Q$ is countable, we only need to show that for any fixed $\varepsilon \in (0, 1)\cap \Q$, 
\begin{align}\label{as-inf-pts}
\sup_{\gamma \in \Aut(\D)}  \# \Big(Z(\frak{f}_\D) \cap \gamma^{-1}(\D_\varepsilon) \Big) = \infty  \quad\text{almost surely}.
\end{align}

Now fix  any $\varepsilon \in (0, 1)\cap \Q$. The distribution of $Z(\frak{f}_\D) \cap \D_\varepsilon$ is given by the determinantal measure induced by the kernel $\chi_{\D_\varepsilon}K_\D \chi_{\D_\varepsilon}$. Since $\rank (\chi_{\D_\varepsilon}K_\D \chi_{\D_\varepsilon}) = \infty$, for any $\ell \in \N$, we have 
\[
\PP(\#(Z(\frak{f}_\D) \cap \D_\varepsilon ) \ge \ell) > 0.
\]
If $\gamma_0 \in \Aut(\D)$ is hyperbolic or parabolic, then Lemma\ref{thm-mixing} implies that the dynamical system $(\Conf(\D), \PP_{K_\D}, \gamma_0)$ is ergodic, whence for any $\ell\in\N$,  the relation
\[
\#(Z(\frak{f}_\D) \cap \gamma_0^{-n}(\D_\varepsilon)) \ge \ell
\]
holds for infinitely many $n$'s on a set of full measure. Since $\ell$ is arbitrary, the desired equality \eqref{as-inf-pts} follows.
\end{proof}

We conclude this section with a conjecture on the asymptotic density of zeros of Gaussian analytic functions.
Let $F$ be a finite subset of the unit circle $\mathbb{T}$ and $\mathfrak{s}_F$ be the corresponding Stolz star domain, which, by definition, is the union, over all $z \in F$, of the Euclidean convex hulls of the unions $\{z\} \cup \{w \in \D:  |w| \le 1/\sqrt{2}\}$ . Let  $\{I_k\}_k$ be the  complementary arcs of the subset $F$ in the unit circle $\mathbb{T}$, and set 
$$\widehat{k}(F): = 1 - \sum_{k} \frac{| I_k|}{2\pi} \log \frac{|I_k|}{2 \pi}.$$
For a countable subset $X\subset \D$ without accumulation points in the interior of the disc, following
   \cite[Chapter 4, Definition 4.9]{HKZ-Bergman}, write 
   \[
D^{+}(X) : = \frac{1}{2}\limsup_{\widehat{k}(F)\to\infty} \frac{\sum_n \{1-|x|^2: x \in \mathfrak{s}_F\cap X \}}{\widehat{k}(F)}, 
\quad D^{-}(X) : = \frac{1}{2}\liminf_{\widehat{k}(F)\to\infty} \frac{\sum_n \{1-|x|^2: x \in \mathfrak{s}_F\cap X \}}{\widehat{k}(F)}.
\]

For $p>1$, let $A^p(\D)$ be the $L^p$-version of Bergman space. Theorem 4.31 and Corollary 4.38 in \cite{HKZ-Bergman} state that 
\[
\text{$X$ is an $A^{2 + \varepsilon}(\D)$-zero set for some $\varepsilon >0$ if and only if $D^{+}(X) < 1/2$;}
\]
\[
\text{$X$ is an $A^{2 - \varepsilon}(\D)$-zero set for all $\varepsilon >0$ if and only if $D^{+}(X) \le 1/2$}.
\] 
\begin{conj}
$D^{+}(Z(\frak{f}_\D))= 1/2$ almost surely.
\end{conj}

%
%\bibliography{mybib}

\begin{thebibliography}{10}

\bibitem{BLPS}
I. Benjamini, R. Lyons, Y. Peres and O. Schramm.
\newblock Uniform spanning forests.
\newblock {\em Ann. Prob.}, 29(1):1--65, 2001.

 \bibitem{BorRains}A.M. Borodin, E.M. Rains.
 \newblock Eynard-Mehta theorem, Schur process, and their pfaffian analogs.
\newblock {\em J. Stat. Phys.}, 121 (2005), 291--317.

\bibitem{Buf-umn}
A.~I. Bufetov.
\newblock Multiplicative functionals of determinantal processes.
\newblock {\em Uspekhi Mat. Nauk.} 67 (2012), no. 1 (403), 177--178;
translation in Russian Math. Surveys 67 (2012), no. 1, 181--182.


\bibitem{Buf-rigid}
A.~I. Bufetov.
\newblock Rigidity of determinantal point processes with the Airy, the Bessel and the gamma kernel.
\newblock {\em Bull. Math. Sci.}  6 (2016), no. 1, 163–172.




\bibitem{BQS}
A.~I. Bufetov.
\newblock {Q}uasi-{S}ymmetries of {D}eterminantal {P}oint {P}rocesses.
\newblock {\em arXiv:1409.2068, to appear in Ann. Probab}.

\bibitem{Buf-inf}
A.~I. Bufetov.
\newblock Infinite determinantal measures.
\newblock {\em Electron. Res. Announc. Math. Sci.}, 20:12--30, 2013.


\bibitem{BDQ-ETDS}
A.~I. Bufetov, Y. Dabrowski and Y. Qiu.
\newblock   {L}inear rigidity of stationary stochastic processes.
\newblock To appear in {\em Ergodic Theory Dynam. Systems}, DOI:10.1017/etds.2016.140.


\bibitem{BFQ-PTRF}
A.~I. Bufetov, S. Fan and Y. Qiu.
\newblock Equivalence of Palm measures for determinantal point processes governed by Bergman kernels.
\newblock {\em  arXiv:1703.08978, to appear in Probab. Th. Rel. Fields}.


\bibitem{BQ-CMP}
A.~I. Bufetov and Y. Qiu.
\newblock   {D}eterminantal point processes associated with {H}ilbert spaces of holomorphic functions.
\newblock {\em  Commun. Math. Phys.}, 351(2017), no.1, 1-44.

\bibitem{BQ-JFA}
A.~I. Bufetov and Y. Qiu.
\newblock {C}onditional measures of generalized {G}inibre point processes.
 \newblock {\em J. Funct. Anal.}, 272 (2017), no. 11, 4671--4708.



\bibitem{Daley-Vere}
D.~J. Daley and D.~Vere-Jones.
\newblock {\em An introduction to the theory of point processes. {V}ol. {II}}.
\newblock Probability and its Applications (New York). Springer, New York,
  second edition, 2008.
\newblock General theory and structure.


\bibitem{RNP-1}
N. Dunford and B.~J. Pettis.
\newblock Linear operations on summable functions.
\newblock{\em Trans. AMS}, 47, (1940). 323--392.




\bibitem{Ghosh-sine}
S. Ghosh.
\newblock Determinantal processes and completeness of random exponentials: the
  critical case.
\newblock {\em Probab. Th. Rel. Fields}, pages 1--23, 2014.


\bibitem{Ghosh-rigid}
S. Ghosh and Y. Peres.
\newblock Rigidity and tolerance in point processes: {G}aussian zeros and
  {G}inibre eigenvalues.
\newblock to appear in {\em Duke Math. J}.



\bibitem{Goldman}
A. Goldman.
\newblock The {P}alm measure and the {V}oronoi tessellation for the {G}inibre process. 
\newblock {\em Ann.
Appl. Probab.}, 20 (2010), 1, 90--128.


\bibitem{HKZ-Bergman}
H. Hedenmalm, B. Korenblum and K. Zhu.
\newblock Theory of Bergman spaces. 
\newblock Graduate Texts in Mathematics, 199, Springer-Verlag, New York, 2000. 

\bibitem{Hocking-Young}
J.~G. Hocking and G.~S. Young.
\newblock {\em Topology}.
\newblock Addison-Wesley, Reading, Mass.-London, 1961.



\bibitem{Soo}
A.~E. Holroyd and T. Soo.
\newblock Insertion and deletion tolerance of point processes.
\newblock {\em Elec. J. Prob.}, 18:no. 74, 2013.




\bibitem{Kallenberg}
O. Kallenberg.
\newblock {\em Random measures}.
\newblock Akademie-Verlag, Berlin; Academic Press, Inc., London, 4th ed., 1986.

\bibitem{DPP-L}
R. Lyons.
\newblock Determinantal probability measures.
\newblock {\em Publ. Math. Inst. Hautes \'Etudes Sci.}, (98):167--212, 2003.

\bibitem{Lyons-ICM}
R. Lyons.
\newblock Determinantal probability: basic properties and conjectures.
\newblock {\em Proc. International Congress of Mathematicians 2014}, Seoul, Korea, vol. IV, 137--161.

\bibitem{Lyons-Zhai}
R. Lyons and A. Zhai. 
\newblock Zero sets for spaces of analysitc functions. 
\newblock to appear in Ann. Inst. Fourier, Grenoble. 


\bibitem{DPP-M}
O. Macchi.
\newblock The coincidence approach to stochastic point processes.
\newblock {\em Adv. Appl. Prob.}, 7:83--122, 1975.

\bibitem{neretin-fock}  Yu. Neretin.
\newblock Determinantal point processes and fermionic Fock space
 \newblock{\em AMS Transl.},  221, 2007, pp.185--191.



\bibitem{Osada-Osada}
H. Osada and S. Osada.
\newblock Discrete approximations of determinantal point processes on continuous spaces: tree representations and tail triviality.
\newblock {\em arXiv:1603.07478, Mar 2016}.


\bibitem{PV-acta}
Yuval Peres and B{\'a}lint Vir{\'a}g.
\newblock Zeros of the i.i.d.\ {G}aussian power series: a conformally invariant
  determinantal process.
\newblock {\em Acta. Math.}, 194(1):1--35, 2005.

\bibitem{Pettis}
B.~J. Pettis.
\newblock  On integration in vector spaces.
\newblock{\em Trans. Amer. Math. Soc.},  44 (1938), no. 2, 277--304.

\bibitem{RNP-Phillips}
R.~S. Phillips.
\newblock On weakly compact subsets of a Banach space.
\newblock{\em Amer. J. Math.,} 65, (1943). 108--136.


\bibitem{pisier-B-martingale}
G. Pisier.
\newblock {\em Martingales in Banach Spaces}.
\newblock Cambridge University Press, Cambridge, 2016.

\bibitem{Rohlin-meas-eng}
V.~A. Rohlin.
\newblock On the fundamental ideas of measure theory.
\newblock {\em Amer. Math. Soc. Transl.}, 1952(71):55, 1952.



\bibitem{Seip-inter-samp} 
K. Seip.
\newblock Interpolation and Sampling in Spaces of Analytic Functions.
\newblock University Lecture Series, 33, American Mathematical Society, Providence, RI (2004), xii+139 pp.

\bibitem{Seip-Fock}
K. Seip. 
\newblock Density theorems for sampling and interpolation in the Bargmann-Fock space I.
\newblock {\em J. reine angew. Math.}, 429 (1992), 91--106.

\bibitem{Seip-beurling-d}
K. Seip.
\newblock Beurling type density theorems in the unit disc.
\newblock{\em Invent. math.}, 113, 21--39(1993).



\bibitem{Seip-d-zero}
K. Seip.
\newblock On Korenblum's density condition for the zero sequences of $A^{-\alpha}$.
\newblock {\em J. Anal. Math.}, 67 (1995), 307--322.

\bibitem {ShirTaka0}
 T. Shirai and Y. Takahashi.
\newblock Fermion process and Fredholm determinant.
\newblock {\em Proceedings of the Second ISAAC Congress}, vol. I, 15--23, Kluwer 2000.


\bibitem{ST-palm}
T. Shirai and Y. Takahashi.
\newblock Random point fields associated with certain {F}redholm determinants.
  {I}. {F}ermion, {P}oisson and boson point processes.
\newblock {\em J. Funct. Anal.}, 205(2):414--463, 2003.

\bibitem{ST-palm2}
T. Shirai and Y. Takahashi.
\newblock Random point fields associated with certain {F}redholm determinants.
  {II}. {F}ermion shifts and their ergodic and {G}ibbs properties.
\newblock {\em Ann. Probab.}, 31(3):1533--1564, 2003.

\bibitem{Simon-det}
B. Simon.
\newblock Notes on infinite determinants of {H}ilbert space operators.
\newblock {\em Adv. Math.}, 24(3):244--273, 1977.

\bibitem{DPP-S}
A. Soshnikov.
\newblock Determinantal random point fields.
\newblock {\em Uspekhi Mat. Nauk}, 55(5(335)):107--160, 2000.


\end{thebibliography}
%\bibliographystyle{plain}
%

\bigbreak
\noindent Alexander I. Bufetov\\
\noindent  Aix-Marseille Universit{\'e}, Centrale Marseille, CNRS, Institut de Math{\'e}matiques de Marseille, UMR7373, \\
  39 Rue F. Joliot Curie 13453, Marseille, France; \\
\noindent Steklov  Mathematical Institute of RAS, Moscow, Russia;\\
\noindent Institute for Information Transmission Problems, Moscow, Russia; \\
\noindent The Chebyshev Laboratory, Saint-Petersburg State University, Saint-Petersburg, Russia.\\
\noindent  bufetov@mi.ras.ru, alexander.bufetov@univ-amu.fr

\bigbreak
\noindent Yanqi Qiu \\
\noindent
Institute of Mathematics, AMSS, CAS and Hua Loo-Keng Key Laboratory of Mathematics, Beijing, China.\\
 CNRS, Institut de Math{\'e}matiques de Toulouse, Universit{\'e} Paul Sabatier, Toulouse, France.\\
\noindent yqi.qiu@gmail.com

\bigbreak
\noindent Alexander Shamov \\
\noindent
Department of Mathematics, Weizmann Institute of Science, Israel.\\
\noindent trefoils@gmail.com

\end{document}